\documentclass[11pt]{amsart}
\usepackage{pgf,tikz,pgfplots,color}
\pgfplotsset{compat=1.15}
\usepackage{mathrsfs} 
\usetikzlibrary{arrows}
\usepackage{fancyhdr}
\usepackage{lastpage} 

\usepackage[margin=1in]{geometry}
\usepackage{geometry}
\usepackage[toc,page]{appendix}
\usepackage[utf8]{inputenc}
\usepackage{hyperref}
\hypersetup{
	colorlinks=true,
	linkcolor=blue,
	citecolor=brown,
	filecolor=magenta,
	urlcolor=blue,
}
\usepackage[notcite,notref]{} 
\usepackage{amsmath}
\usepackage{amsfonts}
\usepackage{amssymb}
\usepackage{amsthm}
\usepackage{setspace}
\usepackage{inputenc}
\usepackage[english]{babel} 
\usepackage{comment}
\usepackage[shortlabels]{enumitem}
\numberwithin{equation}{section}

\title[]{Sobolev $\yh$ estimates for $\dbar$ equations on strictly pseudoconvex domains with $C^2$ boundary}          

\author[]{Ziming Shi} 

\author[]{Liding Yao} 

\address{Department of Mathematics,
	Rutgers University - New Brunswick, Piscataway, NJ, 08854}
\email{zs327@rutgers.edu}

\address{Department of Mathematics,
	University of Wisconsin-Madison, Madison, WI 53706}
\email{lyao26@wisc.edu} 


\newcommand{\und}[1]{{\textcolor{red}{#1}}}

\newcommand{\dist}{\operatorname{dist}}

\newcommand{\supp}{\operatorname{supp}}
\newcommand{\loc}{\mathrm{loc}}

\newtheorem{thm}{Theorem}[section]
\newtheorem{cor}[thm]{Corollary} 
\newtheorem{prop}[thm]{Proposition}
\newtheorem{lemma}[thm]{Lemma}

\theoremstyle{definition}
\newtheorem{defn}[thm]{Definition}
\newtheorem{exmp}[thm]{Example}
\newtheorem{ques}[thm]{Question}

\theoremstyle{remark}
\newtheorem{rem}[thm]{Remark}
\newtheorem*{clm}{Claim}
\newtheorem*{ack}{Acknowledgment}

\renewcommand{\th}[1]{\begin{thm}\label{#1}}
	\renewcommand{\eth}{\end{thm}}
\newcommand{\co}[1]{\begin{cor}\label{#1}}
	\newcommand{\eco}{\end{cor}}
\newcommand{\pr}[1]{\begin{prop}\label{#1}}
	\newcommand{\epr}{\end{prop}}

\newcommand{\df}[1]{\begin{defn}\label{#1}}
	\newcommand{\edf}{\end{defn}}
\newcommand{\ex}[1]{\begin{exmp}\label{#1}} 
	\newcommand{\eex}{\end{exmp}}
\newcommand{\qu}[1]{\begin{ques}\label{#1}}
	\newcommand{\equ}{\end{ques}}  
\newcommand{\mk}{\begin{rem}}
	\newcommand{\emk}{\end{rem}}
\newcommand{\cl}{\begin{clm}}
	\newcommand{\ecl}{\end{clm}} 
\newcommand{\ac}{\begin{ack}}
	\newcommand{\eac}{\end{ack}} 

\newcommand{\ga}{\begin{gather}}
\newcommand{\ega}{\end{gather}}
\newcommand{\gan}{\begin{gather*}}
\newcommand{\egan}{\end{gather*}}
\newcommand{\al}{\begin{gngn}}
	\newcommand{\eal}{\end{align}}
\newcommand{\aln}{\begin{align*}}
\newcommand{\ealn}{\end{align*}}
\newcommand{\eq}[1]{\begin{equation}\label{#1}}
\newcommand{\eeq}{\end{equation}}

\newcommand{\pa}{\partial{}}

\newcommand{\we}{\wedge}

\newcommand{\ra}{\longrightarrow} 
\newcommand{\hra}{\hookrightarrow}

\newcommand{\sm}{\setminus}
\newcommand{\seq}{\subseteq}

\newcommand{\B}{\mathbb{B}} 
\newcommand{\Z}{\mathbb{Z}}
\newcommand{\R}{\mathbb{R}} 
\newcommand{\C}{\mathbb{C}}

\newcommand{\N}{\mathbb{N}}

\newcommand{\V}{\mathcal{V}}
\newcommand{\U}{\mathcal{U}} 

\newcommand{\mc}{\mathcal}
\newcommand{\mf}{\mathfrak}

\newcommand{\ms}{\mathscr}

\newcommand{\mb}{\mathbf}

\newcommand{\tit}{\textit} 
\newcommand{\0}{\mathbf{0}} 
\newcommand{\1}{\mathbf{1}}

\newcommand{\ov}{\overline}
\newcommand{\ti}{\tilde}
\newcommand{\wti}{\widetilde}
\newcommand{\hht}{\widehat}

\newcommand{\RE}{\operatorname{Re}}
\newcommand{\IM}{\operatorname{Im}}
\newcommand{\dbar}{\overline\partial}

\newcommand{\sign}{\operatorname{sign}}

\newcommand{\all}{\alpha}

\newcommand{\del}{\delta}
\newcommand{\Del}{\Delta}
\newcommand{\var}{\varphi}

\newcommand{\ve}{\varepsilon}
\newcommand{\om}{\omega}
\newcommand{\Om}{\Omega}

\newcommand{\La}{\Lambda}
\newcommand{\la}{\lambda}

\newcommand{\gm}{\gamma}
\newcommand{\Gm}{\Gamma}

\newcommand{\yh}{\frac{1}{2}}

\newcommand{\re}[1]{(\ref{#1})}

\newcommand{\rl}[1]{Lemma~\ref{#1}}

\newcommand{\rp}[1]{Proposition~\ref{#1}}
\newcommand{\rt}[1]{Theorem~\ref{#1}}
\newcommand{\rd}[1]{Definition~\ref{#1}}
\newcommand{\rrem}[1]{Remark~\ref{#1}}

\newcommand{\nn}{\nonumber}
\newcommand{\nid}{\noindent}

\newcounter{pp}
\newcommand{\bpp}{\begin{list}{$\hspace{-1em}\alph{pp})$}{\usecounter{pp}}}
	\newcommand{\epp}{\end{list}}

\newcounter{ppp}
\newcommand{\bppp}{\begin{list}{$\hspace{-1em}(\roman{ppp})$}{\usecounter{ppp}}}
	\newcommand{\eppp}{\end{list}}

\newcommand{\Bb}{\mathbb{B}}
\newcommand{\Bs}{\mathscr{B}}

\newcommand{\Df}{\mathfrak{D}}
\newcommand{\Ec}{\mathcal{E}}

\newcommand{\Fs}{\mathscr{F}}

\newcommand{\Gf}{\mathfrak{G}}
\newcommand{\Hc}{\mathcal{H}}
\newcommand{\Kb}{\mathbb{K}}
\newcommand{\Nc}{\mathcal{N}}

\newcommand{\Sc}{\mathcal{S}}
\newcommand{\Ss}{\mathscr{S}}
\newcommand{\Uc}{\mathcal{U}}

\newcommand{\eps}{\varepsilon}
\begin{document}
	\definecolor{rvwvcq}{rgb}{0.08235294117647059,0.396078431372549,0.7529411764705882}

\begin{abstract} 
	We construct a solution operator for $\dbar$ equation that gains $\yh$ derivative in the fractional Sobolev space $H^{s,p}$ on bounded strictly pseudoconvex domains in $\C^n$ with $C^2$ boundary, for all $1 < p < \infty$ and $ s > \frac{1}{p}$. 
	
\end{abstract} 

	\maketitle

	\tableofcontents
	
\section{Introduction}
 The main result of the paper is the following: 
 \begin{thm} \label {thm1_intro}
   Let $\Om$ be a bounded strictly pseudoconvex domain in $\C^n$ with $C^2$ boundary, for $n \geq 2$. Suppose $\var$ is a $\dbar$-closed $(0,q)$ form in $\Om$ where $q \geq 1$. If $\var \in H^{s,p}(\Om)$, for $1 < p < \infty$ and $s > \frac{1}{p}$, 
   then there exists a $(0,q-1)$ form $u$ that solves the equation $\dbar u = \var$, such that $u \in H^{s+ \yh, p}(\Om)$.  Here $H^{s, p}(\Omega)$ is the fractional Sobolev space on $\Om$ (see \rd{Sob_dom_def}). 
 \end{thm}
 
We also prove an analogous result when $\var$ is in Hölder-Zygmund space $\La^r(\Omega)$ which improves an earlier theorem of Gong \cite{Gong19}. 
 
  \begin{thm} \label {thm2_intro} 
Keeping the assumptions of the above theorem. Suppose $\var \in \La^r (\Om)$, where $r > 0$. Then there exists a solution $u$ for $\dbar u = \var$ such that $u \in \La^{r+\yh} (\Om)$. Here $\La^r$ is the H\"older-Zygmund space. 
  	
  \end{thm}
 It is well-known that on a bounded strictly pseudoconvex domain in $\C^n$ with sufficiently smooth boundary, there exist solutions $u$ to the equation $\dbar u = \var$ 
which gains $1/2$ derivative up to boundary if $\var$ belongs to some suitable space. 

In the category of  the $L^2$ Sobolev space (denoted by $H^{s,2}$), one can obtain a solution in the form $\dbar^{\ast} \Nc \var$ where $\Nc$ is the solution operator for the associated $\dbar$-Neumann boundary value problem, and $\dbar^{\ast}$ is the adjoint of $\dbar $ in the $L^2$ Hilbert space. When the boundary  $b\Om$ is $C^{\infty}$, Kohn in his famous work \cite{Koh63} showed that the solution $\dbar^{\ast} \Nc \var$ is in $H^{s+\yh,2}(\Om)$ if $ \var \in H^{s,2}(\Om)$ for any $s \geq 0$. See \cite[Cor.4.4.2, Thm 5.2.6]{C-S01}. 

Later on Greiner and Stein \cite[p.~174]{G-S77}, proved that for any $(0,1)$ form $\var \in H^{k,p}(\Om)$ where $k$ is a non-negative integer and $1 < p < \infty$, $\dbar^{\ast} \Nc \var \in H^{k+ \yh, p}(\Om)$. Their results were later extended by Chang \cite{Chang89} to any $(0,q) $ forms $\var$ for $1 \leq q \leq n$. 
Similarly one can obtain a gain of $\yh$ derivative for the operator $\dbar^{\ast} \mc{N}$ in the Hölder-Zygmund space $\La^r$, for all $r>0$. See \cite[p.~174]{G-S77}. All of these results require that $b\Om \in C^{\infty}$. 
 
Besides the $\dbar$-Neumann approach, one can also solve the $\dbar$ equation on strictly pseudoconvex domains using integral formula with certain ``holomorphic like" kernels. The solutions obtained through this method are no longer $L^2$ canonical,  but have the advantage that boundary no longer needs to be $C^{\infty}$. In this direction,  
Henkin and Ramanov in \cite{H-R71} first constructed a solution which is in $C^{\yh}(\overline\Om)$ if $\var$ is a $(0,1)$ form in the class $C^0 (\overline\Om)$ and $b\Om \in C^2$. Later Siu \cite{Siu74} and Lieb-Range \cite{L-R80} found solutions that are in $C^{k+\yh}(\overline\Om)$ for $\var$ in $C^k(\overline\Om)$, if the boundary is $C^{k+2}$ and $k$ is a positive integer. 
 The requirement on the smoothess of boundary is a result of using integration by parts on certain boundary integral. It is also important to point out that in both papers, the estimates for the solution operators rely on the fact that $\var$ is $\dbar$ closed. 

More recently, Gong \cite{Gong19} used the integral formula method to construct a $\dbar$ solution operator for any $C^2$ strictly pseudoconvex domains, and the solution $u$ lies in $\La^{ r+ \yh}(\Om)$ if $\var$ is any $(0,q) $ form ($q \geq 1$) in the class $\La^r(\Om)$, for all $r>1$. 

In our paper we give a variant of Gong's solution operator which allows one to work on Sobolev spaces when the boundary is $C^2$. Furthermore our operator allows us to obtain $\yh$ estimate when the right-hand side is $\La^r(\Om)$, for all $r > 0$, which improves the above result of Gong. See also \cite{DG19} for estimates on a certain class of weighted Sobolev space.

Here is the outline of the paper: In Section 2 we review the definition and properties of the function spaces we are using. To do estimates we need a characterization of the Sobolev space by Littlewood-Paley theory. We also include some results on interpolation which will be used extensively in our proofs. In Section 3 we recall Rychkov's universal extension operator $E_{\om}$ on a special Lipschitz domain $\om$, whose boundary is the graph of a Lipschitz function. 
Section 4 contains the most technical part of the paper. We show that the commutator $[D,E_{\om}]= DE_{\om}-E_{\om} D$ is in $L^p(\Om, \la)$, where the weight $\la$ is some power of the distance-to-boundary function. In Section 5 we prove various results on the embedding of weighted Sobolev spaces  $W^{k,p}(\Om, \la)$ to $H^{s,p}(\Om)$ spaces. 
Much of the results in this section are probably not new and the procedures are quite routine, although we are unable to find references for the actual results. Section 6 and Section 7 contain the estimates for the homotopy operators which lead to the proof of \rt{thm1_intro} and \rt{thm2_intro}. The main novelty here is the introduction of a weight factor which seems necessary to prove the relevant estimates. 
We mention that the commutator was first introduced by Peters \cite{Pe91} and have been used by Michel \cite{Mi91}, Michel-Shaw \cite{M-S99} among others.

Throughout the paper we assume that all the domains are in $\C^n$ for $n \geq 2$. We denote the set of non-negative integers by $\N $, and the set of positive integers by $\Z^+$. For a set $\Omega\subset\R^N$ we denote $\Omega^c=\R^N\backslash\Omega$. 
We will use the notation $x \lesssim y$ to mean that $x \leq Cy$ where $C$ is a constant independent of $x,y$, and $x \approx y$ for ``$x \lesssim y$ and $y \lesssim x$''. 
For the unit ball in $\R^N$ we use $\B^N$. 
\begin{ack}
 We would like to thank Xianghong Gong for the original inspiration leading to this problem and many helpful discussions. We would like to thank Andreas Seeger for helpful comments about function spaces and the extension operator.  
\end{ack}

\section{Function Spaces and Interpolation}
In this section we review some basic facts about function spaces. 
	\begin{defn}
	Let $\Omega\subseteq\R^N$ be an open set, and let $k\in\N$, $1\leq p<\infty$. We denote by $W^{k,p}(\Om)$ the space of (complex-valued) functions $f\in L^p(\Omega)$ such that $D^\alpha f\in L^p(\Omega)$ for all $|\alpha|\leq k$, with the norm
	\begin{gather*}
	\|f\|_{W^{k,p}(\Omega)}:=\sum_{|\alpha|\leq k}\|D^\alpha f\|_{L^p(\Omega)}=\sum_{|\alpha|\leq k}\left(\int_\Omega|D^\alpha f (x)|^p\right)^\frac1p \, dV(x),\quad 1\le p<\infty.
	\end{gather*}
	
Let $\lambda$ be a positive continuous function on $\Omega$. We define the weighted Sobolev space $W^{k,p}_{\la}(\Omega)$ as the space of $f$ in $W^{k,p}_\loc(\Omega)$ such that the following norm is finite:
	\begin{gather*}
	\| f \|_{W^{k,p}(\Om, \la)}
	:= \sum_{ |\alpha|\leq k}\| \la  D^\alpha f\|_{L^p(\Om)} 
	=\sum_{|\alpha|\leq k}\left(\int_\Omega|D^\alpha f (x)|^p \lambda(x)^p \, dV(x)\right)^\frac1p ,\quad 1\le p<\infty.
	\end{gather*}
	If $\Omega$ is a domain in $\C^n$ with complex variable $z$, we write instead  $\lambda(z)^p dV(z)$. 
\end{defn}
In our application we will take $\lambda(x)=\dist(x,b\Omega)^s$ for some $s\in\R$.
	
We shall use $\Ss(\R^N)$ to denote the space of Schwartz functions, and $\Ss'(\R^N)$ for the space of tempered distributions.  
\begin{defn}\label{Defn::SpecialDomain}
A \textit{special Lipschitz domain} $\omega$ is an open set of $\R^N$ with Lipschitz boundary that has the following form
\[  
	\omega= \{(x',x_N)\in \R^N: x_N>\rho(x')\},
\] 
		where $\rho\in C^{0,1}(\R^{N-1})$ satisfies $\|\rho\|_{C^{0,1}}=\sup_{\R^{N-1}}|\rho|+\sup_{\R^{N-1}}|\nabla\rho|<\frac12$. 
		
	A \textit{bounded Lipschitz domain} $\Omega$ is an bounded open set of $\R^N$ such that for any $p\in b\Omega $ there is an affine linear transformation $\Phi:\R^N\to\R^N$ and a special Lipschitz domain $\omega$ such that
	\begin{equation*}
	    \Omega\cap\Phi(\B^N)=\Phi(\omega\cap\B^N).
	\end{equation*} 
	\begin{defn} \label{Sob_RN_def}
		Let $s\in\R$, $1<p<\infty$. We define $H^{s,p}(\R^N)$ to be the fractional Sobolev space consisting of all (complex-valued) tempered distributions $f\in\Ss'(\R^N)$ such that $(I-\Delta)^\frac s2f\in L^p(\R^N)$, with norm 
		\[ 
		\|f\|_{H^{s,p}(\R^N)} := \|(I-\Delta)^\frac s2f\|_{L^p(\R^N)},
		\]
		where $(I-\Delta)^\frac s2 f$ is given by 
		\[ 
		(I - \Del)^\frac s2 f=((1+4\pi|\xi|^2)^\frac s2 \hht f(\xi))^\vee.
		\]
		Here for a Schwartz function $g$ we set the Fourier transform $\hht g(\xi)=\int_{\R^N} g(x)e^{-2\pi i x \cdot \xi}dx$, and the definition extends naturally to tempered distributions.
\end{defn}
	
\begin{rem}
The Sobolev space $H^{s,p} (\R^N)$ defined above is sometimes called the Bessel potential space. 
There is another type of commonly-used fractional Sobolev spaces called \emph{Sobolev-Slobodeckij spaces}, which is also known as the Besov spaces $\Bs_{pp}^s(\R^N)$, see \cite{HitchhikerFractionalSobolev} for example. We will not use this type of space in our paper with the exception of $B^s_{\infty, \infty}$ which agrees with the H\"older- Zygmund space $\La^s$. 
\end{rem}

We also have the following definition for functions and distributions defined on open sets of $\R^N$: 
\begin{defn}\label{Sob_dom_def} 
Let $\Om \seq \R^N$ be an open set. 
\begin{enumerate}[(i)]
        \item Define $\Ss' (\Om) = \{\tilde f|_{\Om}:\tilde f\in \Ss' (\R^N) \}$. 
        \item For $s \in \R$ and $1 < p < \infty$, define $H^{s,p}(\Om) = \{\tilde f|_{\Om}:\tilde f\in H^{s,p} (\R^N)\}$ with norm
\[
\| f \|_{H^{s,p}(\Om)} = \inf_{\wti{f}|_{\Om} = f} \|\tilde f\|_{H^{s,p} (\R^N)}.  
\] 
\item\label{Item::Sob_dom_def::Domain}  For $s \in \R$ and $1 < p < \infty$, define $H^{s,p}_0 (\Om)$ to be the subspace of $H^{s,p} (\R^N) $ which is the completion of $C^{\infty}_c (\Om)$ under the norm $\| \cdot \|_{H^{s,p} (\R^N)}$. We will write $\|g \|_{H_0^{s,p}(\Omega)}=\|g \|_{H^{s,p}(\R^N)}$ if $g \in H^{s,p}_0(\Om).$
\end{enumerate}
\end{defn} 

\begin{rem}
   In our paper $H_0^{s,p}(\Omega)$ is defined to be the closed subspace of $H^{s,p}(\R^N)$, which is different from some other literature. For example in \cite[Definition 1.95(ii)]{Tr06} Triebel defines the space $\mathring H^{s,p}(\Omega) := \overline{C_c^\infty(\Omega)}^{\|\cdot\|_{H^{s,p}(\Omega)}}$, which is a subspace of $H^{s,p}(\Omega)$.
   
   Nevertheless, when $s > \frac{1}{p}-1$ and $\Om$ is a bounded Lipschitz domain.  
   we have $\mathring H^{s,p}(\Omega)= H^{s,p}_0(\Om)$, in the sense that the natural map $H_0^{s,p}(\Omega)\to \mathring H^{s,p}(\Omega)$ induced by the restriction map $[\tilde f\mapsto\tilde  f|_\Omega]:H^{s,p}(\R^N)\to H^{s,p}(\Omega)$ is a bijection, see equation \eqref{Eqn::Space0Eqv} below. 
\end{rem} 
\begin{rem} \label{Rmk::Hsp=Fp2s} 
  For $1<p<\infty$ and $s\in\R$, the Bessel-Sobolev space $H^{s,p}(\R^N)$ is in fact a special case of the Triebel-Lizorkin space $\Fs_{p2}^s(\R^N)$ with equivalent norm. See \cite[Definition 2.3.1/2 and Theorem 2.5.6(i)]{Tr83}. More precisely we have the following:
\end{rem}

\begin{lemma}[Littlewood-Paley Theorem]\label{Lem::LittlewoodPaleyThm}
   Let $\phi_0\in\Ss(\R^N)$ be a Schwartz function whose Fourier transform satisfies
   \begin{equation*}
       \supp\widehat\phi_0 \seq \{|\xi|<2\},\qquad\widehat\phi_0|_{\{|\xi|\le1\}}\equiv1,\qquad0\le\widehat\phi_0\le 1.
   \end{equation*} 
   For $j\ge1$, let $\phi_j$ be the Schwartz function whose Fourier transform is  
  $\hht{\phi}_0(2^{-j} \xi) - \hht{\phi}_0(2^{-(j-1)} \xi)$.   
  Then for $s\in\R$ and $1<p<\infty$, there exists a $C=C_{\phi_0,p,s}>0$ such that
\begin{equation}\label{Eqn::LPNorm}
C^{-1}\| f \|_{H^{s,p}(\R^N)} 
\leq\bigg(\int_{\R^N}\Big(\sum_{j=0}^\infty2^{2js}|\phi_j\ast f(x)|^2\Big)^\frac p2dx\bigg)^\frac1p
\leq C \|f\|_{H^{s,p}(\R^N)},\quad\forall f\in\Ss'(\R^N),
\end{equation}
provided that either term in the inequality is finite.
\end{lemma} 
	
Note that following the notation from \cite[Section 2.3.1]{Tr83}, we denote the middle term in \eqref{Eqn::LPNorm} by $\|f\|_{\Fs_{p2}^s(\R^N;\phi)}$, which is a Triebel-Lizorkin norm on $\R^N$.
	
	By way of Definition \ref{Sob_dom_def}, one can also define for an arbitrary open set $\Omega \seq \R^N$ the space $\Fs_{p2}^s(\Omega)=\{\tilde f|_\Omega:\tilde f\in \Fs_{p2}^s(\R^N)\}$ equipped with the norm  $\|f\|_{\Fs_{p2}^s(\Omega)}=\inf\limits_{\tilde f|_\Omega=f}\|\tilde f\|_{\Fs_{p2}^s(\R^N)}$ (see \cite[Definition 1.95(i)] {Tr06}. It follows that $H^{s,p}(\Omega)=\Fs_{p2}^s(\Omega)$.

	
In the special case that $s$ is a non-negative integer and $1<p<\infty$, $H^{s,p}$ becomes the familiar Sobolev space $W^{k,p}$. 
\begin{lemma}\label{Lem::Hkp=Wkp}
Let $k\in\N$ and $1<p<\infty$. Then
	\begin{enumerate}[(i)]
		\item \label{Item::Hkp=Wkp::Global} $H^{k,p}(\R^N)=W^{k,p}(\R^N)$ with equivalent norm.
		\item\label{Item::Hkp=Wkp::Local}Let $\Omega$ be a bounded Lipschitz domain in $\R^N$. Then $W^{k,p}(\Omega)=H^{k,p}(\Omega)$ where the norms are equivalent.
	\end{enumerate}
\end{lemma}
\begin{proof}
The proof of \ref{Item::Hkp=Wkp::Global}  can be found in \cite[Theorem 2.5.6(ii)]{Tr83}. 
		
For \ref{Item::Hkp=Wkp::Local}, see \cite[Theorem 1.222(i)]{Tr06}. Notice that we have $H^{k,p}(\Omega)=\Fs_{p2}^k(\Omega)$ as discussed above.
\end{proof}

\begin{rem}
As explained in the proof of \cite[Theorem 1.222(i)]{Tr06}, the key to the proof of  \ref{Item::Hkp=Wkp::Local} is the use of a extension operator $E:W^{k,p}(\Omega)\to W^{k,p}(\R^N)$.
In our paper we need to use a different extension operator that have some nicer properties.
\end{rem}
	
\begin{prop} \label{Prop::DualSpace} 
Let $\Omega\seq\R^N$ be a bounded Lipschitz domain. Suppose $1<p<\infty$ and $s\in\R$. Then we have the following equalities of spaces, where the norms are equivalent.
\begin{enumerate}[(i)]
	\item\label{Item::DualSpace::RN}  $H^{s,p}(\R^N)=H^{-s,p'}(\R^N)'$, where $p'=\frac p{p-1}$.
 \item \label{H_0space} 
 For $s > \frac{1}{p}-1$, $H^{s,p}_0( \Om) 
 = \{f \in H^{s,p}(\R^N): f|_{\overline\Omega^c}  = 0 \}$.  
	\item\label{Item::DualSpace::H0Char} $H^{s,p}_0(\Omega)
	=H^{-s,p'}(\Omega)'$ and $H^{-s,p'}(\Omega)=H^{s,p}_0(\Omega)'
	$, provided that $s > \frac{1}{p}-1$.
	\end{enumerate}
\end{prop}
\begin{proof} 
For proof of \ref{Item::DualSpace::RN} see \cite[Theorem 2.6.1(a)]{Tr95}. 

The proof of \ref{H_0space} and \ref{Item::DualSpace::H0Char} are the combination of several results in \cite{TriebelLipschitzDomain}. We now offer some explanations.
Recall in Remark \ref{Rmk::Hsp=Fp2s} we can use $H^{s,p}=\Fs_{p2}^s$ for $s\in\R$ and $1<p<\infty$. 

In \cite[Section 3.2]{TriebelLipschitzDomain}, Triebel defines $\widetilde H^{s,p}(\overline{\Omega})\subseteq H^{s,p}(\R^N)$ and $\widetilde H^{s,p}(\Omega)\subseteq H^{s,p}(\Omega)$  as
\begin{equation}\label{Eqn::HtildeSpace}
    \widetilde H^{s,p}(\overline\Omega):=\{f\in H^{s,p}(\R^N):\supp f\subseteq\overline\Omega\},\quad \widetilde H^{s,p}(\Omega):=\{f|_\Omega:f\in\widetilde H^{s,p}(\overline{\Omega})\}.
\end{equation}
So $\{f\in H^{s,p}(\R^N):f|_{\overline\Om^c}=0\}=\widetilde H^{s,p}(\overline\Omega)$.

Clearly $\widetilde H^{s,p}(\overline\Omega)$ (resp. $\widetilde H^{s,p}(\Omega)$) is a closed subspace of $H^{s,p}(\R^N)$ (resp. $H^{s,p}(\Omega)$), and we have a surjective restriction map $[f\mapsto f|_\Omega]:\widetilde H^{s,p}(\overline\Omega)\to \widetilde H^{s,p}(\Omega)$.

When $s>\frac1p-1$, by \cite[Proposition 3.1]{TriebelLipschitzDomain} we have $\widetilde H^{s,p}(\overline\Omega)=\widetilde H^{s,p}(\Omega)$ in the sense that the restriction map $f\mapsto f|_\Omega$ is bijective.  

Recall that by definition $H_0^{s,p}(\Omega)=\overline{C_c^\infty(\Omega)}^{\|\cdot\|_{H^{s,p}(\R^N)}}$ is a closed subspace of $H^{s,p}(\R^N)$. Also observe that $H_0^{s,p}(\Omega)\subseteq\widetilde H^{s,p}(\overline{\Omega})$, since if $f=\lim_{j\to\infty} f_j$ and $\supp f_j\seq\Omega$, then $\supp f\subseteq\overline\Omega$. Thus $H_0^{s,p}(\Omega)=\overline{C_c^\infty(\Omega)}^{\|\cdot\|_{\widetilde H^{s,p}(\overline{\Omega})}}$.

By \cite[Theorem 3.5(i)]{TriebelLipschitzDomain}, for $s>\frac1p-1$, $C_c^\infty(\Omega)$ is dense in $\widetilde H^{s,p}(\Omega)$. Hence by using the identification $\widetilde H^{s,p}(\overline{\Omega})=\widetilde H^{s,p}(\Omega)$, we get for $s>\frac1p-1$,
\begin{equation}\label{Eqn::Space0Eqv}
    H_0^{s,p}(\Omega) =\overline{C_c^\infty(\Omega)}^{\|\cdot\|_{\widetilde H^{s,p}(\overline{\Omega})}}=\overline{C_c^\infty(\Omega)}^{\|\cdot\|_{\widetilde H^{s,p}(\Omega)}}=\widetilde H^{s,p}(\Omega).
\end{equation}
Using $\widetilde H^{s,p}(\overline\Omega)=\widetilde H^{s,p}(\Omega)$ we obtain $H_0^{s,p}(\Omega)=\{f\in H^{s,p}(\R^N):f|_{\overline\Om^c}=0\}$, which proves \ref{H_0space}. 

By \cite[Definition 3.3 and (43)]{TriebelLipschitzDomain}, we have duality $H^{-s,p'}(\Omega)=\wti H^{s,p}(\Omega)'$ and $\wti H^{s,p}(\Omega)=H^{-s,p'}(\Omega)'$ when $s >\frac1p-1$, where the norms are equivalent. Using the identification $\widetilde H^{s,p}(\Omega)=H^{s,p}_0(\Omega)$, we get $H^{s,p}_0(\Omega)=H^{-s,p'}(\Omega)'$ and $H^{-s,p'}(\Omega)= H^{s,p}_0(\Omega)'$ for the given range of $s$, proving \ref{Item::DualSpace::H0Char}. 
\end{proof}


We also need some interpolations results. 
\begin{defn}
	Let $X_0,X_1$ be two Banach spaces that belong to a larger ambient space. For $0<\theta<1$. The  \emph{interpolation space $[X_0,X_1]_\theta$} is defined to be the space consisting of all $f(\theta)\in X_0+X_1$, where $f:\{z\in\C:0\leq \RE z\leq 1\}\to X_0+X_1$ is a continuous map that is analytic in the interior, such that $f(it)\in X_0$ and $f(1+it)\in X_1$ for all $t\in\R$. The norm is given by
	\[  
	\|u\|_{[X_0,X_1]_\theta}=\inf_{f} \{\sup\limits_{t \in\R}(\|f(it)\|_{X_0}+\|f(1+it)\|_{X_1}):u=f(\theta)\}. 
	\]  
\end{defn}

\begin{prop}[Complex interpolation theorem] \label{cinterpol} 
	Let $X_0,X_1,Y_0,Y_1$ be Banach spaces that belong to some larger ambient spaces, and Suppose $T: X_0+X_1\to Y_0+Y_1$ is a linear operator such that for each $i=0,1$,  
	$\|Tu\|_{Y_i}\leq C_0\|u\|_{X_i}$ for all $u\in X_i$. Then $T:[X_0,X_1]_\theta \to [Y_0,Y_1]_\theta$ is bounded linear with $\|Tu\|_{[Y_0,Y_1]_\theta}\leq C_0^{1-\theta}C_1^\theta\|u\|_{[X_0,X_1]_\theta}$ for all $u\in[X_0,X_1]_\theta$.
\end{prop} 
	See \cite[Theorem 1.9.3(a) and Definition 1.2.2/2]{Tr95}.
	
We also have the following facts: 
\begin{prop}\label{interpol_space}
Let $\Omega$ be an open set of $\R^N$. Let $1<p<\infty$ and $s_0,s_1\in\R$. Denote $\delta(x)=\dist(x,b\Omega)$ and set $s_\theta=\theta s_1+(1-\theta)s_0$ for $0<\theta<1$. Then the following hold: 
\begin{enumerate}[(i)] 
		\item\label{Item::cinterpolsp::WeiLp} $[L^p(\Omega,\delta^{s_0}),L^p(\Omega,\delta^{s_1})]_\theta = L^p(\Omega,\delta^{s_\theta})$.
		\item\label{Item::cinterpolsp::Hsp}  $[H^{s_0,p}(\Om),H^{s_1,p}(\Om)]_\theta=H^{s_\theta,p}(\Om)$, provided that $\Om$ is a bounded Lipschitz domain. 
	\end{enumerate}
\end{prop}

The proof of \ref{Item::cinterpolsp::WeiLp} can be found in \cite[Theorem 1.18.5]{Tr95}. The proof of \ref{Item::cinterpolsp::Hsp} can be found in \cite[Corollary 1.111 (1.372)]{Tr06}. 
	
\section{The Universal Extension operator}
In this section we recall the construction of the universal extension operator by Rychkov \cite{Ry99}. None of the results here is new, although we shall present the proof in a way slightly different from Rychkov's.

In the rest of the paper we shall denote by $\Kb$ the positive cone in $\R^N$: 
\begin{equation*}
 	\Kb=\{(x',x_N):x_N>|x'|\}.
\end{equation*}
\end{defn}
\begin{rem}
   In many literature, for example \cite[Section 1.11.4 (1.322) p.~63]{Tr06}, the definition for a special Lipschitz domain only requires $\rho$ to be a Lipschitz function. In other words,  $\|\nabla\rho\|_{L^\infty(\R^{n-1};\R^{n-1})}$ is finite but can be arbitrary large.
   By taking invertible linear transformation we can make $\nabla\rho$ small in new coordinates. 
\end{rem}

	\begin{defn}\label{Defn::DyaRes}
		A \emph{regular dyadic resolution} is a sequence $\phi=(\phi_j)_{j=0}^\infty$ of Schwartz functions, denoted by $\phi\in\mf{D}$, such that
		\begin{itemize}
			\item $\int\phi_0=1$, $\int x^\alpha\phi_0(x)dx=0$ for all $\alpha\in\N^N\backslash\{0\}$.
			\item $\phi_j(x)=2^{nj}\phi_0(2^jx)-2^{n(j-1)}\phi_0(2^{j-1}x)$, for $j\ge 1$.
		\end{itemize}
		
		A \emph{generalized dyadic resolution} is a sequence $\psi=(\psi_j)_{j=0}^\infty$ of Schwartz functions, denoted by $\psi\in\mf{G }$, such that
		\begin{itemize}
			\item  $\int x^\alpha\psi_1(x)dx=0$ for all $\alpha\in\N^N$.
			\item $\psi_j(x)=2^{n(j-1)}\psi_1(2^{j-1}x)$, for $j\ge1$.
		\end{itemize}
	Here $\psi_0$ can be an arbitrary Schwartz function.
	\end{defn}
	
	\begin{lemma}[{\cite[Theorem 4.1(a)]{Ry99}}] \label{Lem::k_char}
There exists a function $g\in\Ss(\R)$ such that $\supp g\subseteq[1,\infty)$, $\int_\R g=1$ and $\int_\R t^kg(t)dt=0$ for all $k\in\Z^+$.
	\end{lemma}
	\begin{proof}
		Define
		\begin{gather*}
		G(z):=\exp(-(z-1)^\frac18-(z-1)^{-\frac18}),\quad z \in\C\backslash[1,\infty).  
		\end{gather*}
Here we use $(z-1)^\frac18=|z-1|^{\frac18\operatorname{arg}(z-1)}$ with  $0 <\operatorname{arg}(z-1)<2\pi$.  It is easy to check that the two branches $G(t+i0)$ and $G(t-i0)$ are both smooth functions which are flat at $t=1$. 
		
		For $0<\ve <\frac12$, take an oriented loop $\Gamma_\ve \seq\C$ with 
		\begin{align*}
		\Gm_\ve &= \{t+i0:1+\eps\leq t\leq \ve^{-1}\}  \cup \{\ve ^{-1}e^{i\theta}:0\leq\theta\leq 2\pi\}
		\\ &\quad \cup\{-t-i0:- \ve^{-1}\leq t \leq -1-\eps\}\cup\{1+ \ve e^{-i\theta}:-2\pi\leq \theta\leq 0\}. 
		\end{align*}  
		By Cauchy's theorem,
		\begin{align} \label{con_int}  
		\frac1{2\pi i}\int_0^\infty t^k(G(t+i0)-G(t-i0))dt
		& =\lim\limits_{\ve \to0^+}\frac1{2\pi i}\int_{\Gamma_\ve}z^kG(z)dz 
		\\ \nn &= 
		\begin{cases}
		G(0)\neq0,&k=-1  \\0, & k\geq 0. 
		\end{cases}  
		\end{align}
		Define 
		\[
		g(t):= \frac{1}{(2 \pi i )G(0) }\frac{G(t+i0)-G(t-i0)}{t},\quad t\in\R.  
		\]
Then $g \equiv 0$ on $(-\infty,1)$. Also $g$ vanishes to infinite order at both $t=\infty$ and $ t=1$. In view of \eqref{con_int}, we have
\[ 
	\int_0^\infty g(t) \, dt =1, \quad \int_0^\infty t^kg(t)dt=0, \quad \forall \:  k\in\Z^+.\qedhere 
\]
\end{proof}
\begin{lemma}[{\cite[Proposition 2.1]{Ry99}}]\label{Lem::k_dyadic} 
Recall $-\Kb=\{x_N<-|x'|\}$ and let $\Df,\Gf$ be given in Definition \ref{Defn::DyaRes}.
\begin{enumerate}[(i)]
	\item\label{Item::k_dyadic::Phi} There is a $\phi=(\phi_j)_{j=0}^\infty\in\mf{D}$ on $\R^N$ such that $\supp\phi_j\subseteq-\Kb\cap\{x_N<-2^{-j}\}$ for all $j\in\N$.		\item\label{Item::k_dyadic::Psi} For any $\phi=(\phi_j)$ satisfying  \ref{Item::k_dyadic::Phi}, there is a $\psi=(\psi_j)_{j=0}^\infty\in\mf{G}$ such that $\supp\psi_j\subseteq-\Kb\cap\{x_N<-2^{-j}\}$ for all $j\in\N$ and $f=\sum_{j=0}^\infty\psi_j\ast\phi_j\ast f$ for all $f\in\Ss'(\R^N)$.
		\end{enumerate}
	\end{lemma}
	\begin{defn}\label{Note::k_dyadic}
	We call $(\phi,\psi)=(\phi_j,\psi_j)_{j=0}^\infty$ with above-mentioned properties a \emph{$\Kb$-dyadic pair}.
	\end{defn}
	\begin{proof}[Proof of Lemma \ref{Lem::k_dyadic}]  
		Let $g\in\Ss(\R)$ be as in Lemma \ref{Lem::k_char} which is supported in $[1,\infty)$. Take an invertible linear transformation $\Theta=(\theta_1,\dots,\theta_N):\R^N\to\R^N$ such that $\Theta^{-1}([1,\infty)^N)\subseteq-\Kb\cap\{x_N<-1\}$. Define
		\[
		\phi_0(x_1,\dots,x_N) = C_0 \,   g(\theta_1(x))\cdots g(\theta_N(x)),  
		\]
		where $C_0\neq0$ is the constant chosen so that $\int_{\R^N}\phi_0=1$, or $\hht{\phi}_0 (0) = 1$. Then $\phi_0\in\Ss(\R^N)$ satisfies $\supp\phi_0\seq\Theta^{-1}([1,\infty)^N)$.
		Moreover, $\phi_0$ satisfies $\int_{\R^N} x^{\all} \phi_0(x) \, dx = 0$ for all $|\all| >0$ since $\int t^kg(t)dt=0$ for all $k\in\Z^+$.
		
		Define $\phi_j (x)  =2^{Nj}\phi_0(2^j x) - 2^{N (j-1)}\phi_0 (2^{j-1} x )$ for $j\geq 1$, so then $\supp \phi_j \subseteq \{ x_N < -2 ^{-j} \} \cap -\Kb$.  This proves \ref{Item::k_dyadic::Phi}.

To prove \ref{Item::k_dyadic::Psi}, let 
\[
	\rho_0:=\phi_0\ast \phi_0\in\Ss(\R^N), \quad \rho_j (x): =2^{Nj}\rho_0 (2^j x )-2^{N(j-1)}\rho_0(2^{j-1} x), \quad j \geq 1. 
\]
Then $\supp \rho_0  \seq \supp \phi_0 + \supp \phi_0 \seq -\Kb \cap  \{ x_N < -2 \}$ and therefore
\[
    \supp \rho_j \seq \{ x_N \leq -2 \cdot 2^{- (j-1)} \} \cap -\Kb  = \{ x_N <  -2^{-j} \} \cap -\Kb.
\]
		
		So $\rho\in \Df$ satisfies $\supp\rho_j\seq-\Kb\cap\{x_N<-2^{-j}\}$ for all $j\ge0$ and $\hht\rho_j(\xi)= \hht \phi_j(\xi)(\hht \phi_0(2^{-j}\xi)+\hht \phi_0(2^{-( j-1)}\xi))$ for $j\ge1$. Therefore
		\begin{align*}
		1&=\sum_{j,k=0}^\infty\hht\rho_j(\xi)\hht \rho_k(\xi)
		\\ & =\sum_{j=0}^\infty \hht\rho_j(\xi)\Big(\hht \rho_j(\xi)+2\sum_{k=j+1}^\infty \hht\rho_k(\xi)\Big)
		\\ &=\sum_{j=0}^\infty \hht \rho_j(\xi)\big(\hht \rho_j(\xi)+2-2 \hht \rho_0(2^{-j}\xi)\big) 
		\\ &= \hht\rho_0(\xi)(2- \hht\rho_0(\xi))+\sum_{j=1}^\infty \hht\rho_j(\xi)(2- \hht \rho_0(2^{-j}\xi)- \hht \rho_0(2^{- (j-1)}\xi))
		\\ &= [\hht \phi_0(\xi)]^2(2- \hht\rho_0(\xi))+\sum_{j=1}^\infty \hht \phi_j(\xi)(\hht\phi_0(2^{-j}\xi) - \hht\phi_0(2^{-(j-1)}\xi))(2- \hht \rho_0(2^{-j}\xi)- \hht\rho_0(2^{-( j-1) }\xi).
		\end{align*}
		We can now define $\psi$ via its Fourier transform as 
		\begin{align*} 
		\textstyle\hht \psi_0(\xi)&:=2 \hht  \phi_0(\xi)- \hht  \phi_0(\xi)^3 
		\\ 
		\hht  \psi_j(\xi)&:=( \hht  \phi_0(2^{-j}\xi) -  \hht  \phi_0(2^{-(j-1)}\xi))(2 - \hht  \rho_0(2^{-j}\xi)- \hht  \rho_0(2^{-(j-1)}\xi)), \quad j\ge1.
		\end{align*}
		Then $\sum_{j=0}^{\infty} \hht{\phi}_j  \hht{\psi}_j= 1$. Note that $\hht{\psi}_j (\xi ) = \hht{\psi}_1 (2^{-(j-1) } \xi  )$ for $j\ge1$, and therefore 
		\[
		\psi_j (x) = 2^{N  (j-1) }  \psi_1   (2^{j-1} x),\quad j\ge1.  
		\]
		Also we have
		\[
		\psi_j (x) =  \left( 2^{Nj} \phi_0 (2^j x) - 2^{N (j-1)} \phi_0  (2^{j-1} x )  \right) \ast \left(  2 \del_0 - 2^{N j} \rho_0  (2^j x) - 2^{N (j-1)} \rho_0 (2^{j-1} x) \right),\quad j\ge1 .  
		\]
		Since $\supp \phi_0$ and $ \supp \rho_0$ are contained in $ -\Kb \cap \{ x_N < -1 \}$, we have $\supp \psi_j  \seq - \Kb \cap \{ x_N < -2^{-j} \}$. 
		Also we get $\hht \psi_1(\xi)=O(|\xi|^\infty)$ from $\hht  \phi_0(\xi)=1+O(|\xi|^\infty)$, which implies $\int x^\alpha\psi_1(x)dx=0$ for all $\all$ with $|\all | \geq 0$. 
	\end{proof} 
	
We can now define the universal extension operator, first for special Lipschitz domains, and then for bounded Lipschitz domains. 
	
\begin{defn}\label{ext_opt_def}
Let $(\phi,\psi)$ be a $\Kb$-dyadic pair, and let $\om$ be a special Lipschitz domain. The universal extension operator $E_{\om}$ associated with $(\phi,\psi)$ is defined by 
\begin{equation} \label{Eom_def}
	E_{\om} f:=\sum_{j=0}^\infty\psi_j\ast(\mb{1}_{\om}  \cdot(\phi_j\ast f)), 
	\end{equation}
where $\mb{1}_{\om}$ is the characteristic function on $\om$. 
		
\end{defn} 

Here by extension, we mean for any tempered distribution $f\in\Ss'(\omega)$, $(Ef)|_\omega=f$ as distributions on $\om$. Indeed since $\omega+\Kb=\omega$, we have $(\psi_j\ast(\1_\omega h))|_\omega=h|_\omega$ for $h\in L^1_\loc(\R^N)$. Thus
$$(E_\omega f)|_\omega=\sum_{j=0}^\infty(\psi_j\ast(\1_\omega(\phi_j\ast f)))|_\omega=\sum_{j=0}^\infty(\psi_j\ast\phi_j\ast f)|_\omega=f.$$ 
	

More generally for a bounded Lipschitz domain $\Om$, and $\U$ an open set containing $\overline{\Omega}$, we can use partition of unity to define extension operator $\Ec=\Ec_\Omega$ for $\Om$ such that $\supp \Ec\var\subset \U$ for all $\var$: Let $\{U_{\nu}\}_{\nu=0}^M$ be a finite open cover of $\overline{\Om}$, such that $U_{0} \subset\subset \Om$, $b \Om \seq \bigcup_{\nu=1}^M U_{\nu}$ and $\bigcup_{\nu=0}^MU_\nu\subseteq\U$.
Furthermore we may assume that for each $\nu$, there exists a special Lipschitz domain $\om_{\nu}$ and an invertible affine linear transformation $\Phi_{\nu}: \R^N \to \R^N$, such that $U_{\nu} = \Phi_{\nu} (\B^{N})$ and $U_{\nu} \cap \Phi_{\nu} (\om_{\nu}) = U_{\nu} \cap \Om $.

Choose $\chi_{\nu} \in C^{\infty}_c (U_{\nu})$ such that $\chi_0 + \sum_{\nu=1}^M \chi_{\nu}^2 \equiv 1$ in some neighborhood of $ \Om$. For a function $g$ defined on $\Om \cap U_{\nu}$, let $E_{\nu} g = E_{\om_{\nu}} (g \circ \Phi_{\nu}) \circ \Phi_{\nu}^{-1}$. 
 
\begin{equation} \label{E_def}
   \Ec f = \chi_0 f + \sum_{\nu = 1}^M \chi_{\nu} E_{\nu} (\chi_{\nu} f ). 
\end{equation}

\begin{prop} \label{Prop::ext_op}
	Let $\omega$ be a special Lipschitz domain, and $E_{\om}$ be given by \eqref{Eom_def}. 
	Then 
	\begin{enumerate}[(i)] 
	    \item\label{Item::ext_op::Sob} $E_{\omega}: H^{s,p}(\omega)\to H^{s,p}(\R^N)$ 
	is a bounded operator for all $s \in \R$ and $1<p<\infty$. 
	\item\label{Item::ext_op::Hold} $E_{\omega}: \La^s(\omega)\to \La^s(\R^N)$ 
	is a bounded operator for all $s>0$.
	\end{enumerate}
\end{prop}
The reader can find the proof in \cite[Theorem 4.1(b)]{Ry99}. By partition of unity we see that $\Ec:H^{s,p}(\Omega)\to H^{s,p}(\R^N)$ and $\Ec:\La^s(\Omega)\to \La^s(\R^N)$ are also bounded operators.

There is also a useful ``Littlewood-Paley type" characterization of $H^{s,p}(\om)$. 
\begin{prop} \label{Prop::OmegaEqvNorm} 
Let $\om$ be a special Lipschitz domain and $\phi = (\phi_j)_{j=0}^{\infty}$ be constructed as in \rl{Lem::k_dyadic} \ref{Item::k_dyadic::Phi}. 
\begin{enumerate}[(i)]
    \item\label{Item::OmegaEqvNorm::Sob} For $s\in\R$ and $1<p<\infty$, $H^{s,p}(\omega)$ has equivalent norm
  \begin{equation*}
      \| f \|_{\Fs_{p2}^{s} (\om; \phi)} := \Big\|\Big(\sum_{j=0}^\infty 2^{2js} |\phi_j\ast f|^2\Big)^\frac12\Big\|_{L^p(\omega)}.
  \end{equation*} 
  \item\label{Item::OmegaEqvNorm::Hold} For $s>0$, the H\"older-Zygmund space $\La^s(\omega)$ has an equivalent norm 
 \begin{equation*}
 \|f\|_{\Bs_{\infty\infty}^s(\omega; \phi) }:=\sup_{j\in\N}2^{js}\|\phi_j\ast f\|_{L^\infty(\omega)}.
 \end{equation*}
\end{enumerate}
\end{prop}

The proof is in \cite[Theorem 3.2]{Ry99}, where the assumption is $\phi_0\in C_c^\infty(-\Kb)$, but based on \cite[Theorem 4.1(b)]{Ry99}  same proof works for $\phi_0\in\Ss(-\Kb)$. 
		

\section{Commutator Estimate}

The main result of this section is the following commutator estimate on special Lipschitz domains. 
We will write $D$ for the gradient operator, and $D^k = (D^{\all})_{|\alpha|=k}$. 

\th{commest_thm} 
Let $1< p < \infty$ and $s\in\R$, and let $\omega$ be a special Lipschitz domain. Suppose $(\phi,\psi)$ is a $\Kb$-dyadic pair and let $E_{\om}$ be defined as in Definition \ref{ext_opt_def}. Then there exists a constant $C_{p,s}>0$ such that for $\delta(x)=\dist(x,b\omega)$,
\begin{equation} \label{commest}
   \|\delta^{1-s}[D,E_{\om}] f\|_{L^p(\overline\omega^c) }\leq C\| f \|_{H^{s,p} (\om)},\quad\forall f \in\Ss'(\R^N).
\end{equation}
\eth 
\begin{rem} \label{commest_rem} 
	
\
\begin{enumerate}[(i)]
\item \label{commest_rem1}
By \rp{Prop::OmegaEqvNorm}, the $H^{s,p}(\om)$ norm is equivalent to the $\Fs^s_{p,2} (\om; \phi)$ norm. In fact we will prove the following stronger estimate: for $s\in\R$, $1\le p\le\infty$, 
\begin{equation}\label{Eqn::commest_rem::StrongEst}
    \|\delta^{1-s}[D,E_{\om}] f\|_{L^p(\overline\omega^c) }\leq C_{s,p,\phi}\| f \|_{\Fs^s_{p, \infty}(\om;\phi)},\quad\forall f \in\Ss'(\R^N),
\end{equation}
where
\[
    \|f\|_{\Fs^s_{p,\infty} (\omega;\phi)}:=\Big\|\sup_{j\in\N}2^{js}|\phi_j\ast f|\Big\|_{L^p(\omega)} \leq\Big\|\Big(\sum_{j=0}^\infty2^{2js}|\phi_j\ast f|^2\Big)^\frac12\Big\|_{L^p(\omega)}= \| f \|_{\Fs^s_{p,2}(\om; \phi)}. 
\]

\item \label{commest_rem2}

When $f \in H^{s,p} (\om)$ for $s > 1$, Theorem \ref{commest_thm} follows from Proposition \ref{C3}, which gives that
\eq{gwest} 
  \| \del^{1-s}  g  \|_{L^p(\overline\omega^c)} \leq C \|  g \|_{H^{s-1,p}(\R^N)}, \quad \text{for any $g \in H^{s-1,p}_0 (\overline\omega^c)$. }  
\eeq
This is because $[D, E_{\om}] \equiv 0$ in $\om$, and therefore $[D,E_{\om}]f \in H^{s-1,p}_0(\overline\om^c)$ by \rp{Prop::DualSpace} \ref{H_0space}. Letting $g = [D, E_{\om}] f$ in \re{gwest} we obtain \eqref{commest}. 

\item \label{commest_rem3}
When $s>\frac1p$ and $f\in H^{s,p}(\omega)$, Theorem \ref{commest_thm} implies that $[D,E_\omega]f \in L^1_\loc(\R^N)$. Indeed, since $[D,E_\omega] f$ is supported in $\om^c$, 
\begin{align*}
    \|[D,E_\omega] f\|_{L^1(B^N(0,R))}\le&\|\delta^{s-1}\|_{L^{p'}(B^N(0,R))}\|\delta^{1-s}[D,E_\omega] f\|_{L^p(B^N(0,R))}
    \\\le&\|\delta^{s-1}\|_{L^{p}(B^N(0,R))}\|\delta^{1-s}[D,E_\omega] f\|_{L^p(\overline{\omega}^c)}
    \\
    \lesssim&\|\delta^{(s-1)p'}\|_{L^{1}(B^N(0,R))}^{1/p'}\|f\|_{H^{s,p}(\omega)}. 
\end{align*}
for every $R$. 

Note that $(s-1)p'>(\frac1p-1)\frac p{p-1}=-1$, so $\delta^{(s-1)p'}$ is a locally integrable in $\R^N$, which implies that the right-hand side is finite.
When $s\le\frac1p$, $\delta(x)^{(s-1)p'}$ is no longer integrable near the boundary of $\omega$, and we can only interpret the commutator as a distribution. 
\end{enumerate} 
\end{rem} 

 Note that when $p = \infty$ and $s>0$, we have $\sup\limits_{j\in\N}2^{js}\|\phi_j\ast f\|_{L^\infty(\omega)}=\big\|\sup\limits_{j\in\N}2^{js}|\phi_j\ast f|\big\|_{L^\infty(\omega)}$, or $ \Bs_{\infty,\infty}^s(\omega;\phi)=\Fs^s_{\infty,\infty} (\omega;\phi)$ (also see \cite[Remark 2.3.4/3]{Tr83}). Thus by \eqref{Eqn::commest_rem::StrongEst} and Proposition \ref{Prop::OmegaEqvNorm} \ref{Item::OmegaEqvNorm::Hold}, we have the following estimate for the H\"older-Zygmund space:
 \begin{cor}\label{Cor::Comm_Hold}Let $s>0$ and let $E_\omega,\omega$ be as in Theorem 4.1. There is a $C>0$ such that
 \begin{equation} \label{commest_Holder}
  \| \del^{1-s}  [D, E_{\om}] f  \|_{L^{\infty}(\overline\omega^c)} 
  \leq C \| f \|_{\La^s (\om)}, \quad f\in\La^s(\om). 
\end{equation}
 \end{cor}


To prove \rt{commest_thm} we need a sequence of lemmas.  
		
		\begin{lemma}\label{convo_le}
			Let $\phi, \psi$ be two generalized dyadic resolutions. Then for any $M >0$ and $\gm \in \N^N$, there is a $C = C_{M. \gm} > 0$ such that 
			\eq{convo_est}
			\int_{|x| \geq 2^{-l}}  \left| D^{\gm} \phi_j \ast \psi_k (x) \right| \, dx  
			\leq C 2^{\min(j,k)  | \gm| - M (|j-k|  + \max (j,k) - l ) } ,\quad j,k\ge0,\quad l\in\Z. 
			\eeq
		\end{lemma} 
		\begin{proof}
		By symmetry we can assume $j \leq k$. We first use the scaling properties of $\phi$ and $\psi$ to show that the estimate can be reduced to the cases $j= 0$ and $j=1$. When $1<j\le k$, recall that $\phi_j (x) = 2^{N(j-1)} \phi_1 (2^{j-1} x)$ and $\psi_k(x)=2^{N(k-1)}\psi_1(2^{k-1}x)$, so
		\begin{align*}
		\phi_j \ast \psi_k(x)
		&=2^{N(j+k-2)}\int\phi_1 (2^{j-1} x-2^{j-1}y)\psi_1(2^{k-1}y)dy
		\\&= \int\phi_1 (2^{j-1} x-\tilde y)\psi_1(2^{k-j}\tilde y)d\tilde y
		\\ &= 2^{N(j-1)}\phi_1\ast\psi_{k-j+1}(2^{j-1}x).
		\end{align*}
		Therefore taking substitution $\tilde x=2^{j-1}x$ we have
		\begin{align}\label{Eqn::convo_le::convo_est0}  
		\int_{|x| \geq 2^{-l}} \left| D^{\gm} \phi_j \ast \psi_k (x) \right|dx= 2^{(j-1) | \gm| }  \int_{|\tilde x| \geq 2^{j-1-l} } |D^{\gm} \phi_1 \ast \psi_{k-j+1}(\tilde x)|d\tilde x,\quad 1\le j\le k.
		\end{align} 
		
	Suppose \eqref{Eqn::convo_le::convo_est0} is true for $j=1\le k$. Since $k \geq j$, the right hand side of \eqref{Eqn::convo_le::convo_est0} is bounded by
\begin{align*}
   C 2^{(j-1) |\gm| }  2^{|\gm| - M(|1- ( k-j+1)| +  \max(1, k-j+ 1) + (j-1 -l )) }
   &=C 2^{j|\gm| - M(|j-k|+k-l)}
   \\& =C 2^{j|\gm| - M(|j-k|+\max(j,k)-l)}.
\end{align*}
	This proves the reduction. 
		
		Next we consider the case for $j \in \{ 0, 1\}$ and $k\ge 1$. Write $k = 1 + m$, for $m\ge 0$. Since $\int _{\R^N} x^{\all}  \psi_1 (x) = 0$ for any $\all\in\N^N$, we have
		\begin{align*}  
		D^{\gm} \phi_j  \ast \psi_{m+1} (x) 
		&=  \int_{\R^N} D^{\gm} \phi_j (x-y) \psi_{m+1} (y) \, dy
		\\ &= \int_{\R^N} D^{\gm} \phi_j (x-y) 2^{Nm} \psi_{1} (2^m y) \, dy 
		\\ &= \int_{\R^N} D^{\gm} \phi_j (x-2^{-m} y) \psi_{1} ( y) \, dy  
		\\ &= \int_{\R^N} \left(  D^{\gm} \phi_j (x- 2^{-m} y) - \sum_{|\all| \leq M'-1} (-2^{-m}y)^{\all} 
		\frac{D^{(\gm + \all)} \phi_j (x)}{\all!} \right)  \psi_1 (y)  \, dy ,
		\end{align*}    
		where $M'$ is some large number to be chosen. By Taylor's theorem, the expression in parenthesis is bounded in absolute value by  
		\[
		\frac{1}{M'!} | 2^{-m} y |^{M'} \sup_{B(x, 2^{-m} |y| )}  \left| D^{|\gm | + M'} \phi_j \right| .
		\]
		Since $\phi_0$ and $\phi_1$ are Schwartz, for $j=0,1$ we have  
		\begin{equation}\label{Eqn::convo_le::BddTmp}
		\sup\limits_{B(x,2^{-m}|y|)}\left| D^{|\gm | + M'} \phi_j \right|  \lesssim_{\gamma,M'}
		\begin{cases}
		(1+ |x|)^{- M'}, & |x| \geq 2^{1-m} |y| 
		\\  1,  & |x| < 2^{1-m} |y|
		\end{cases}\quad \text{for }M'>0,\gamma\in\N^N.
		\end{equation}
		Therefore for $j =0$ or $1$ we have 
		\begin{align*}  
		&\int_{|x|  \geq 2^{-l} } \left|  D^{\gm} \phi_j  \ast \psi_{m+1} (x) \right| \, dx
		\\ & \quad \leq   \int_{|x|  \geq 2^{-l}} \, dx \left( \int_{|y| \leq 2^{m-1} |x| } + \int_{|y| \geq 2^{m-1}|x|}  \right)    \frac{| 2^{-m} y |^{M'} }{M'!} \sup_{B_x( 2^{-m} |y| )}  \left| D^{|\gm | + {M'}} \phi_j \right|   |\psi_1 (y) |  \, dy
		\\ & \quad \leq \int_{|x|  \geq 2^{-l}} \left( \int_{|y| \leq 2^{m-1}  |x| } | 2^{-m} y |^{M'} (1+ |x|)^{-{M'}} |\psi_1(y) | \, dy + \int_{|y| \geq 2^{m-1}  |x|} | 2^{-m} y |^{M'} | \psi_1 (y) | \, dy \right)  \, dx.
		\end{align*}   
		
		Using polar coordinates and \eqref{Eqn::convo_le::BddTmp} the above is bounded by 
		\begin{align*}  
		&2^{-m{M'}} \int_{2^{-l}} ^{\infty} \rho^{N-1} \, d \rho  \bigg( (1+ \rho)^{-{M'}}   \int_0^{\infty}  r^{M'} r^{N-1} (1+r)^{-2{M'} -N} \, dr
		\\
		&\qquad + \int_{2^{m-1} \rho}^{\infty} r^{M'} (1+r)^{-2{M'} -N} r^{N-1} \, dr \bigg)   
		\\\lesssim&   2^{-m {M'}}   \int_{2^{-l}} ^{\infty}  \rho^{N-1}  \left(  (1 + \rho)^{- {M'}} + \int_{2^{m-1} \rho}^{ \infty}  (1+ r)^{-{M'}-1} \, dr \right)  \, d \rho
		\\\lesssim& 2^{-m{M'}}  \int_{2^{-l}}^{\infty} \rho^{N-1} \left( ( 1+ \rho )^{-{M'}} + (1+ 2^{m-1} \rho)^{-{M'}} \right) \, d  \rho
		\\\lesssim&  2^{-m{M'}}  \int_{2^{-l}}^{\infty} \rho^{N-1} ( 1+ \rho )^{-{M'}}  \, d  \rho.
		\end{align*}   
		Taking ${M'} \geq 2M + N$, then the right-hand side is bounded by $2^{-2Mm}  2^{Ml} = 2 ^{-M(m+m- l)}$, which is what we need for the estimate. 
\end{proof}

\begin{cor}\label{Cor::PsiDecay}
Let $(\psi_j)_{j=0}^\infty\seq\Ss( \R^N)$ be a generalized dyadic decomposition. Then for any $M>0$ and $\gamma\in\N^N$ there is a $C=C_{\psi,M,\gamma}>0$ such that
\begin{equation}\label{Eqn::PsiDecay}
\int_{|x|\ge2^{-l}}|D^\alpha\psi_k(x)|dx\le C 2^{k|\gamma|-M\max(0,k-l)},\quad\forall  k\in\N,\quad l\in\Z.
\end{equation}
\end{cor} 
\begin{proof}
	Let $\phi=(\phi_j)_{j=0}^\infty$ be any regular dyadic resolution, so $\psi_k=\sum_{j=0}^\infty \phi_j\ast\psi_k$ and we have $\int_{|x|\ge2^{-l}}|D^\alpha\psi_k|\le\sum_{j=0}^\infty\int_{|x|\ge2^{-l}}|D^\alpha (\phi_j\ast\psi_k)|$. Taking sum over $j\in\N$ on the right hand side of \eqref{convo_est} we get \eqref{Eqn::PsiDecay}. 
\end{proof}
		We remark that Corollary \ref{Cor::PsiDecay} can also be proved independently without the use of \rl{convo_le}. 
		
		\medskip
		In our proof we use the following dyadic decomposition: 
		\begin{gather} \label{Pk}
		P_k:=\{(x', x_N):2^{-\frac12-k}< x_N-\rho(x')<2^{\frac12-k}\} \seq \om, \quad  k\in\Z, 
		\\ \label{Sk}
		S_k:=\{(x',x_N):-2^{\frac12-k}<x_N -\rho(x')<-2^{-\frac12-k}\} \seq \overline\omega^c,\quad  k\in\Z. 
		\end{gather} 
	Up to sets of measure zero, we have disjoint unions $\om = \bigsqcup_{k \in \Z} P_k $ and $\overline\omega^c = \bigsqcup_{k \in \Z} S_k$.

		\begin{lemma} 
			For any $M$ there is a $C_M >0$ such that for every $ j,j'  \in \N$ and $k,k' \in \Z$, 
			\eq{diag_est} 
			\|\psi_j\ast(D \mb 1_\om \cdot(\phi_j\ast\psi_{j'}\ast(\mb{1}_{P_{k'}}\cdot(\phi_{j'}\ast f)))\|_{L^p(S_k)}
			\leq C2^{j-M(|j-k|+|j-j'|+|j'-k'|)}\|\phi_{j'}\ast f\|_{L^p(P_{k'})}.
			\eeq 
		\end{lemma}  
		\begin{proof}	
		Since $D \mb{1}_{\om} = - D \mb1_{\overline\omega^c}$ as distributions on $\R^N$, the term $\psi_j\ast(D \mb 1_\om \cdot(\phi_j\ast\psi_{j'}\ast(\mb{1}_{P_{k'}} \cdot(\phi_{j'}\ast f)))$ in \eqref{diag_est} has two identical expressions: 
		\begin{align} \label{comm_exp1}
		A_{jj'k'}& := D\psi_j\ast(\1_\om(\phi_j\ast\psi_{j'}\ast(\1_{P_{k'}}(\phi_{j'}\ast f))))-\psi_j\ast(\1_\om  ( D\phi_j\ast\psi_{j'}\ast(\1_{P_{k'}}(\phi_{j'}\ast f)))),
		\\ \label{comm_exp2}
		B_{jj'k'}& := -D\psi_j\ast(\1_{\omega^c}(\phi_j\ast\psi_{j'}\ast(\1_{P_{k'}}(\phi_{j'}\ast f))))+\psi_j\ast(\1_{\omega^c}(D\phi_j\ast\psi_{j'}\ast(\1_{P_{k'}}(\phi_{j'}\ast f)))).
		\end{align}
		
		First we show that the left-hand side of \eqref{diag_est} is $0$ if $j \leq k-2$ or $j' \leq k'- 2$. By \eqref{comm_exp1} and the fact that $\supp \phi_j, \supp \psi_j  \seq \{ x_N < -2^{-j}\} $, we see if $j' \leq k'-2$, 
		\begin{align*}
		\supp ( \phi_j\ast\psi_{j'}\ast(\1_{P_{k'}}\cdot(\phi_{j'}\ast f)) )
		&\seq \supp \phi_j + \supp \psi_{j'} + \supp \left( \1_{P_{k'}}\cdot(\phi_{j'}\ast f)  \right) 
		\\&\seq \{ x_N - \rho  (x') < - 2^{1-j}  - 2^{1-j'} + 2^{-k'}  \}
		\\ &\seq \{ x_N - \rho  (x') < 0 \}, 
		\end{align*}
		and similarly,  
		\begin{align*}
		\supp(D\phi_j\ast\psi_{j'}\ast(\1_{P_{k'}}\cdot(\phi_{j'}\ast f)) ) 
		\seq \{ x_N - \rho  (x') < 0 \}. 
		\end{align*}
		Hence the left-hand side of \eqref{diag_est} vanishes if $j' \leq k'-2$. On the other hand, the expression in  \eqref{comm_exp2} is supported in $ \supp \psi_j + \overline\omega^c  \seq \{  x_N - \rho (x') < - 2^{-j} \}$, which is disjoint from $S_k = \{ -2^{-k+ \yh} < x_N - \rho (x') < -2^{-k - \yh} \}$ when $j < k-1$. Hence the left-hand side of \eqref{diag_est} again vanishes. We shall now assume that $j \geq k-2$ and $ j' \geq k'-2$. 
		
		We first estimate the left-hand side of \eqref{diag_est} using \eqref{comm_exp1}. Write \eqref{comm_exp1} as $A_{jj'k'}=:A^1_{jj'k'}-A^2_{jj'k'}$ where
	\begin{gather*}
	  A^1_{jj'k'} :=D\psi_j\ast(\1_\om(\phi_j\ast\psi_{j'}\ast(\1_{P_{k'}}(\phi_{j'}\ast f)))),
	  \\ 
	  A^2_{jj'k'} := \psi_j\ast(\1_\om  ( D\phi_j\ast\psi_{j'}\ast(\1_{P_{k'}}(\phi_{j'}\ast f)))).    
	\end{gather*}

	Denoting $h_{jj'k'} := \phi_j\ast\psi_{j'}\ast(\1_{P_{k'}}\cdot(\phi_{j'}\ast f))$, we have 
		\[
		\| A^1_{jj'k'} \|_{L^p(S_k)} = \| D \psi_j \ast (\1_{\om} \cdot h_{jj'k'}) \|_{L^p (S_k)}. 
		\]
		For $ x \in S_k$ and $y \in \om$, we have $|x-y|  \geq \dist (S_k, \om) \ge 2^{-1- k}$, and
		\[
		D \psi_j \ast (\1_{\om} h_{jj'k'} ) (x) = \int_{\om  \cap |x-y| \geq 2^{-1-k}} D \psi_j (x-y) h_{jj'k'}(y) \, dy.  
		\] 
		Since $ \psi_j = 2^{n(j-1)} \psi_1( 2^{j-1} x ) $ and $\psi_1$ and $D \psi_1$ are Schwartz functions, we have
		\begin{gather} \label{dpsi_est}    
		\int_{|x| > 2^{-1-k}} | D \psi_j (x)| \,dx = \int_{|x| > 2^{j-2-k}} 2^{j-1} |D  \psi_1(x) | \, dx
		\lesssim_M 2^{j-M(j-k) } 
		\\  \label{psi_est} 
		\int_{|x| > 2^{-1-k}} |\psi_j (x) | \, dx = \int_{|x| > 2^{j-2 -k}}  | \psi_1 (x) |  \lesssim_M 2^{-M (j-k)}. 
		\end{gather}
		In view of \rl{convo_le} and $j \geq k-2$, we have 
		\[
		\| D^{\gm} \phi_j \ast \psi_{j'} \|_{L^1(\R^N)}  \lesssim 2^{j |\gm| -M |j-j'|}. 
		\]
		Applying Young's inequality and using \eqref{dpsi_est}  
		\begin{align*} 
		\| A^1_{jj'k'} \|_{L^p (S_k)}
		&\leq \| D \psi_j \|_{L^1 (\{ |x| \geq 2^{-1- k} \})} \| h_{jj'k'} \|_{L^p (\om)}   
		\\ &\lesssim 2^{j- M(j-k)} \| \phi_j\ast\psi_{j'}\ast(\1_{P_{k'}}\cdot(\phi_{j'}\ast f)) \|_{L^p(\om)}
		\\ &\lesssim 2^{j- M(j-k)} \| \phi_j \ast \psi_{j'} \|_{L^1 (\R^N)} \| \phi_{j'} \ast f \|_{L^p (P_{k'})}  
		\\ &\lesssim 2^{j} 2^{-M|j-j'|} 2^{-M (j-k)} \| \phi_{j'} \ast f \|_{L^p (P_{k'})}. 
		\end{align*}
		Similarly by \eqref{psi_est} and \rl{convo_le} we can show that  
		\[
		\| A^2_{jj'k'}  \|_{L^p (S_k)}  \lesssim 2^{j} 2^{-M|j-j'|} 2^{-M (j-k)} \| \phi_{j'} \ast f \|_{L^p (P_{k'})}. 
		\]
		
		Next, we estimate the left-hand side of \eqref{diag_est} using \eqref{comm_exp2}. Write \eqref{comm_exp2} as $B_{jj'k'}=:-B^1_{jj'k'}+B^2_{jj'k'}$ where
		\begin{gather*}
		    B^1_{jj'k'}:=D\psi_j\ast(\1_{\omega^c}(\phi_j\ast\psi_{j'}\ast(\1_{P_{k'}}(\phi_{j'}\ast f)))),\\
		    B^2_{jj'k'}:=\psi_j\ast(\1_{\omega^c}(D\phi_j\ast\psi_{j'}\ast(\1_{P_{k'}}(\phi_{j'}\ast f)))).
		\end{gather*}
		Denoting $h'_{j'k'} = \1_{P_{k'}}  \cdot ( \phi_{j'} \ast f)$ and applying Young's inequality we get
		\eq{B1}
		\| B^1_{j'k'} \|_{L^p (S_k)} \lesssim \| D \psi_j \|_{L^1 (\R^N)} \| \phi_j \ast \psi_{j'}  \ast h'_{j'k'} \|_{L^p (\overline\omega^c)}. 
		\eeq 
		For $x \in \overline\omega^c $ and $y  \in P_{k'}$, we have $|x-y| \geq \dist (P_{k'}, \overline\omega^c) \ge 2^{- 1-k'}$. Hence for $x \in \overline\omega^c$, 
		\begin{align*}
		D^{\gm} \phi_j \ast \psi_{j'} \ast h'_{j'k'}(x)
		= \int_{P_{k'} \cap \{ y: |x-y|  \geq 2^{-1 - k'} \}}  (D^{\gm} \phi_j \ast \psi_{j'} ) (x-y)h'_{j'k'}(y) \, dy. 
		\end{align*}
		By Young's inequality and \rl{convo_le},    
		\begin{align*}
		\|   D^{\gm} \phi_j \ast \psi_{j'} \ast h'_{j'k'}\|_{L^p (\overline\omega^c)} 
		&\leq \left( \int_{|x|  \geq 2^{-1 - k'}} | D^{\gm} \phi_j \ast \psi_{j'} |  \right) \| h'_{j'k'} \|_{L^p (\R^N)} 
		\\ &\lesssim  2^{j |\gm| - M (|j-j'| + j' - k') }  \| \phi_{j'} \ast  f \|_{L^p (P_{k'})}.      
		\end{align*}
		Since $ \| D \psi _j \|_{L^1 (\R^N)}  = 2^{j-1} \| D \psi_1 \|_{L^1 (\R^N)} \leq C2^{j-1}$, we get \eqref{B1} that  
		\[
		\| B^1_{jj'k'} \|_{L^p (S_k)} \lesssim 2^{j-1-M|j-j'| -M(j'-k')} \| \phi_{j'} \ast  f \|_{L^p (P_{k'})}.  
		\]
		In the same way we can show that $B_2$ satisfies the same estimate 
		\[
		\| B^2_{jj'k'} \|_{L^p (S_k)} \lesssim 2^{j-M|j-j'| -M(j'-k')} \| \phi_{j'} \ast  f \|_{L^p (P_{k'})}.   
		\]
		Finally combining the estimates for $A^1_{jj'k'}$, $A^2_{jj'k'}$, $B^1_{jj'k'}$ and $B^2_{jj'k'}$ we get
		\begin{align*}
		&\|\psi_j\ast((D \mb 1_\om) \cdot(\phi_j\ast\psi_{j'}\ast(\mb{1}_{P_{k'}}\cdot(\phi_{j'}\ast f)))\|_{L^p(S_k)} 
		\\ &\quad \leq 2^{j - M|j-j'| } \min \left\{  2^{-M (j-k)} , 2^{-M (j'-k')}\right\} \| \phi_{j'} \ast f \|_{L^p (P_{k'})} 
		\\ &\quad \leq 2^{j-M |j-j'| }  2^{-M \max\{ j-k, j'-k' \}} \| \phi_{j'} \ast f  \|_{L^p (P_{k'})}  
		\\&\quad \leq 2^{j-M |j-j'| }  2^{- \frac{M}{2}[( j-k) + (j'-k')] } \| \phi_{j'} \ast f  \|_{L^p (P_{k'})}   
		\\ &\quad \lesssim 2^{j- \frac M2 \left[ | j-j'| + |j-k| + |j'-k'| \right]  } \| \phi_{j'} \ast f  \|_{L^p (P_{k'})}, 
		\end{align*}
		where in the last step we use $ j \geq k-2$ and $j' \geq k'-2$. Replacing $M/2$ by $M$ we get the result.
		\end{proof}
		
		\begin{lemma}  \label{sumle} 
			Let $M > 1$. Then there is a $C_M >0$ such that   
			\[
			\sum_{b \in \Z}  2^{-M (|a-b| + |b-c| )} \leq C_M 2^{- (M-1) | a -c | }, \quad \text{for all $a, c \in \Z$.}     
			\] 
		\end{lemma}  
		\nid  \tit{Proof.}   
		By a substitution $\wti{b} = a-b$ we see that 
		$
		\sum_{b \in \Z}  2^{-M (|a-b| + |b-c| )} 
		= \sum_{\wti{b} \in \Z}  2^{-M(\wti{b} + |a-c-\wti{b}| }$. 
		So it suffices to show that 
		\[
		\sum_{b \in \Z}  2^{-M(b + |a-b |) } \leq C_M 2^{- (M-1) | a | }.   
		\]
		By symmetry we can assume $a>0$ in the above inequality.  It follows that 
		\begin{align*} 
		\sum_{b \in \Z}  2^{-M(b + |a-b |) } 
		&\leq  \sum_{  \text{ $b \leq 0$  or $b \geq  2a$}}  2^{-M(|b| + |a-b| )}  + \sum_{b=1}^{a} 2^{-M (b + (a-b))}
		+ \sum_{b=a+1}^{2a-1}  2^{-M (b+ (b-a))} 
		\\ &\leq  \sum_{b \in \Z} 2^{-M (|b|+ a)} + \sum _{b=1}^{a-1} 2^{-M a} + \sum_{b=a+1}^{2a-1} 2^{-Ma } 
		\\&= 2^{-Ma}  \left(  \sum_{b \in \Z} 2^{-M |b| } +  \sum_{b=1}^{2a-1} 1 \right) 
		\\&\lesssim (a+1)  2^{-Ma}.
		\end{align*}
		Clearly $a+1 \leq 2^a$, for $a \in \Z^+$. Hence $  \sum_{b \in \Z}  2^{-M(b + |a-b |) } \leq 2^{- (M-1)  |a|}$.  
		\hfill \qed
		
We are now ready to prove the main result of the section. 
\begin{proof}[Proof of \rt{commest_thm}]In view of \rrem{commest_rem} \ref{commest_rem1}, we will prove the following stronger estimate 
\begin{equation}\tag{\ref{Eqn::commest_rem::StrongEst}}
    \|\delta^{1-s}[D,E_{\om}] f\|_{L^p(\overline\om^c) }\leq C_{p,s}\big\|\sup\limits_{j\in\N}2^{js}|\phi_j\ast f|\big \|_{ L^p(\omega)},\quad\forall f \in\Ss'(\om),
\end{equation}
for $1\le p\le\infty$, provided that the norm on the right hand side is finite.

	Let $P_k$ and $S_k$ be the dyadic strips defined in \eqref{Pk} and \eqref{Sk}. Since $\del \approx 2^{-k}$ on $S_k$, we can replace the function $\del$ by $ \sum_{k \in \Z} 2^{-k} \1_{S_k}$. Also 
		\begin{align*}
		[D, E_\omega] f &= DE_\omega f - E_\omega Df  
		\\ &= \sum_{j=0}^{\infty} \psi_j \ast \left[ (D \1_{\om})  \cdot (\phi_j  \ast f) + \1_{\om} \cdot (\phi_j  \ast Df) \right] - \sum_{j=0}^{\infty} \psi_j \ast \left( \1_{\om} \cdot (\phi_j  \ast Df) \right)
		\\&=  \sum_{j=0}^{\infty} \psi_j \ast \left( (D \1_{\om})  \cdot (\phi_j  \ast f)  \right) 
		\\&= \sum_{j,j'=0}^{\infty} \psi_j \ast \left( (D \1_{\om})  \cdot \left(\phi_j  \ast \psi_{j'} \ast ( 1_\omega \cdot (\phi_{j'} \ast f) )\right)  \right) 
		\\&=  \sum_{j,j' \in \N, k' \in \Z} \psi_j \ast \left( (D \1_{\om})  \cdot (\phi_j  \ast \psi_{j'} \ast ( 1_{P_{k'}} \cdot (\phi_{j'} \ast f) ))  \right).
		\end{align*}
		Denoting the summand on the right-hand side by $A_{jj'k'}$, we have by \rl{diag_est},   
		\[
		\| A_{jj'k'} \|_{L^p (S_k)}  \lesssim 2^{j- M \left( | j-j'| + |j-k| + |j' -k'|\right) } \| \phi_{j'} \ast  f \|_{L^p (P_{k'})}.  
		\] 
		Therefore
		\begin{align*}
		\|  \del^{1-s}[D, E_\omega ] f \|_{L^p(S_k)} 
		&\lesssim 2^{-k(1-s)} \left\| [D, E_\omega ] f  \right\|_{L^p(S_k)}  
		\\ &\lesssim 2^{(s-1)k}  \sum_{j,j' \in \N, k' \in \Z} 2^{j- M \left( | j-j'| + |j-k| + |j' -k'|\right) } \| \phi_{j'} \ast  f \|_{L^p (P_{k'})}. 
		\end{align*}
		Write $k= j'+ (k-j')$ and $ j= j' + (j-j')$. Then the above is bounded by  
		\begin{align*}
		\|  \del^{1-s}[D, E_\omega ] f \|_{L^p(S_k)} 
		&\lesssim \sum_{j,j' \in \N, k' \in \Z} 2^{(s-1) j' + j'} 2^{| s-1| |k-j'| + |j-j'|}  2^{- M \left( |j-k| + | j-j'| + |j' -k'|\right) } \| \phi_{j'} \ast  f \|_{L^p (P_{k'})}
		\\&\lesssim \sum_{j,j' \in \N, k' \in \Z} 2^{(|s-1| + 1)  (|j-j'| + |k-j| )} 2^{- M \left( |j-k|+ | j-j'| + |j' -k'|\right) } 
		2^{sj'} \| \phi_{j'} \ast  f \|_{L^p (P_{k'})}
		\\ & \lesssim \sum_{j,j' \in \N, k' \in \Z}  2^{- (M - |s-1| -1) \left( |j-k| + | j-j'| + |j'-k'| \right) } 2^{sj'} \| \phi_{j'} \ast  f \|_{L^p (P_{k'})}.  
		\end{align*}
		Applying \rl{sumle} to the sum over $j$ and then again to the sum over $j'$, we get   
		\begin{align*}
		\|  \del^{1-s}[D, E_\omega ] f \|_{L^p(S_k)} 
		&\lesssim \sum_{j'\in\N, k' \in \Z} 2^{- (M - |s-1| - 2 ) (| k-k'|+|j'-k'|)  } 2^{sj'} \| \phi_{j'} \ast  f \|_{L^p (P_{k'})} 
		\\ &\lesssim \sum_{ k' \in \Z} 2^{- (M - |s-1| - 3 ) | k-k'|  } \sup_{j' \in \N} 2^{sj'} \| \phi_{j'} \ast  f \|_{L^p (P_{k'})}  
		\\ &\lesssim  \sum_{ k' \in \Z} 2^{- (M - |s-1| - 3 ) | k-k'|  } \| \sup_{j' \in \N} 2^{sj'} |\phi_{j'} \ast  f | \|_{L^p (P_{k'})}. 
		\end{align*} 
		
		Define sequences $u, v, w$ by 
		\[
		u[j] := \|  \del^{1-s}[D, E_\omega ] f \|_{L^p(S_j)}, \quad  
		v[j] := 2^{-( M-|s-1| -3 ) |j|}, \quad    
		w[j] := \| \sup_{l \in \N} 2^{s l} |\phi_{l} \ast  f | \|_{L^p (P_{j})}. 
		\]
		Then we have shown that $u \lesssim v \ast w $. By Young's inequality we get $\| u \|_{\ell^p} \leq \| v \|_{\ell^1} \| w \|_{\ell^p}$. Clearly $\| u \|_{\ell^p} \approx \| \del^{1-s} [D, E_\omega ] f \|_{L^p (\overline\omega^c)}$ and $ \| w \|_{\ell^p} = \|  f \|_{\ms{F}^{s}_{p, \infty}(\omega;\phi)}$ (see Remark \ref{commest_rem} \ref{commest_rem1}). By choosing $M$ sufficiently large so that $\| v \|_{\ell^1}<\infty$ we obtain the desired estimate \eqref{Eqn::commest_rem::StrongEst}. 
\end{proof}

For a bounded Lipschitz domain, let $E$ be an extension operator defined by  \eqref{E_def},
\[
     E f = \chi_0 f + \sum_{\nu = 1}^M \chi_{\nu} E_{\nu} (\chi_{\nu} f ), 
\]
where $E_{\nu} g = E_{\om_{\nu}} (g \circ \Phi_{\nu}) \circ \Phi_{\nu}^{-1}$. For our application we need to glue together the commutator estimates on each special Lipschitz domain. We now end the section with a lemma which will help us achieve this goal. 
\begin{lemma} \label{Enu_commest_le}Let $\Omega\subseteq\R^N$ be a bounded Lipschitz domain, and let $E_\nu$ be defined as above. Then for $s\in\R$ and $1<p<\infty$, we have 
\begin{equation*} 
  \| \del^{1-s}[D, E_{\nu}] (\chi_{\nu} f )  \|_{L^p (U_{\nu} \cap \overline\Omega^c)}  \leq C_{s,p}\| f \|_{H^{s,p} (\Om)},\quad\forall f\in H^{s,p}(\Omega). 
\end{equation*}
\end{lemma}
\begin{proof}
We have
\begin{align*}
 [D, E_\nu] ( \chi_\nu f )
&= D (E_\nu (\chi_\nu f )) - E_\nu (D (\chi_\nu f )) 
\\ & = D( E_{\om_\nu} [ (\chi_\nu f) \circ \Phi_\nu] \circ \Phi_\nu^{-1})
- E_{\om_\nu} [D (\chi_\nu f) \circ \Phi_\nu ] \circ \Phi_\nu^{-1}
\\ &= (D  E_{\om_\nu}[(\chi_\nu f) \circ \Phi_\nu] ) \circ \Phi_\nu^{-1} \cdot D \Phi_\nu^{-1}  
- E_{\om_\nu}[D ( (\chi_\nu f) \circ \Phi_\nu)] \circ \Phi_\nu^{-1} \cdot D\Phi_\nu^{-1} 
\\ &\quad + E_{\om_\nu}[D ( (\chi_\nu f) \circ \Phi_\nu)] \circ \Phi_\nu^{-1} \cdot D\Phi_\nu^{-1} - E_{\om_\nu} [D (\chi_\nu f) \circ \Phi_\nu ] \circ \Phi_\nu^{-1}
\\ &= [D, E_{\om_\nu}] ( (\chi_\nu f) \circ \Phi_\nu) \circ \Phi_\nu^{-1} \cdot D\Phi_\nu^{-1}. 
\end{align*}

Note that in the last step we used the fact that $\Phi_{\nu}$ is a linear transformation so that 
\begin{align*} 
E_{\om_{\nu}}[D ( (\chi_{\nu} f) \circ \Phi_{\nu})] \circ \Phi_{\nu}^{-1} \cdot D\Phi_{\nu}^{-1} 
&= E_{\om_{\nu}} [  D (\chi_{\nu} f)  \circ \Phi_{\nu} ] \circ \Phi_{\nu}^{-1}  \cdot D \Phi_{\nu} \cdot D \Phi_{\nu}^{-1}
\\ &= E_{\om_{\nu}} [  D (\chi_{\nu} f)  \circ \Phi_{\nu} ] \circ \Phi_{\nu}^{-1}.  
\end{align*}
Applying \rp{commest_thm} to the domain, we have
\begin{align*} 
  \| \del^{1-s}[D, E_{\nu}] (\chi_{\nu} f )  \|_{L^p (U_{\nu} \cap \overline\Omega^c)}  
 &{\lesssim} \| \del^{1-s}[D, E_{\om_{\nu}}] (\chi_{\nu} f \circ \Phi_{\nu})  \|_{L^p (\Phi_{\nu}^{-1} (U_{\nu} \cap \overline\Omega^c))}  
 \\ &= \| \del^{1-s}[D, E_{\om_{\nu}}] (\chi_{\nu} f \circ \Phi_{\nu})  \|_{L^p (\B^{N} \cap \overline\om_{\nu}^c)}  
 \\& {\lesssim} \| \chi_{\nu} f \circ  \Phi_{\nu} \|_{H^{s,p} (\om_{\nu})} 
 \\& {\lesssim} \| f \|_{H^{s,p} (\Om)},  
\end{align*}
where we used that $U_\nu=\Phi_\nu(\B^N)$ and $\Phi_{\nu} (\B^N\cap \om_{\nu}) = U_{\nu} \cap \Om$. 
\end{proof}

\section{Hardy-Littlewood Lemma of Sobolev Type}  

In the last section we estimated $L^p(\Om, \la)$ norm using $H^{s,p}(\Om)$ norm, where the weight $\la$ is some power of the boundary distance function. To show that our solution for $\dbar$ is in $H^{s+\yh,p}(\Om)$, we also need to bound the $H^{s,p}(\Om)$ norms by weighted Sobolev norms $W^{k,p}(\Om, \la)$. We will call this kind of estimates Hardy-Littlewood lemma of Sobolev type, after the classical version for H\"older spaces. 

\begin{lemma}\label{wtibple}  
	Let $s > -\frac{1}{p}$ and $1 < p < \infty$. Then
	\begin{enumerate}[(i)]
	    \item\label{Item::wtibple::1}  There is a $C_{s,p}>0$ such that for all $v\in W^{1,p}_\loc(0,2)$ such that $v \equiv 0$ near $2$,
	\eq{ldwtibp}
	\int_0 ^2  t^{sp} | v (t) |^p  \, dt \leq C_{s,p} \int_0^2 t^{(s+1) p} |v'(t)|^p \, dt.  
	\eeq
	\item\label{Item::wtibple::2} Let $\om =\{x_N>\rho(x')\}\seq \R^N$ be a bounded special Lipschitz domain. Suppose $u \in W^{1,p}_{\loc}(\om)$ and $\supp u \subseteq \ov{\om} \cap \B^N$. Then 
\[ 
	\| \del^s u \|_{L^p (\om)} \leq   C_{s,p}' \| \del^{s+1} Du \|_{L^p (\om)}, 
\]
	where $\del (x) = \dist (x, b \om)$ and $C_{s,p}'>0$ is the constant that does not depend on $u$. 
	\end{enumerate}
\end{lemma} 
\begin{proof}
\ref{Item::wtibple::1} By assumption $v\in W^{1,1}_\loc(0,2)$ is locally absolutely continuous, hence $v(t)$ can be defined point-wise.  

Let $\ve$ be a small positive number. Using integration by parts we have
\begin{align*} 
\int_{\del}^2 t^{sp} | v(t) |^p  \, dt 
&= - \frac{\del^{sp+1}}{sp+1} + \int_{\del}^2 \frac{t^{sp+1}}{sp+1} p |v(t)|^{p-1} v'(t) \sign(v(t)) \, dt 
\\ & \leq \int_{\del}^2 \frac{t^{sp+1}}{sp+1} p | v(t) |^{p-1} |v'(t) |  \, dt 
\\ &\leq C_{s,p} \| t^{s (p-1)} |v|^{p-1}  \|_{L^{\frac{p}{p-1}} ([\del,2])} \| t^{s+1}  v' \|_{L^p([\del,2])} 
\\ &=C_{s,p}  \| t^s v \|_{L^p ([\del,2])}^{p-1} \| t^{s+1}  v ' \|_{L^p ([\del,2])}.
\end{align*}
Here $\sign x=\frac x{|x|}$ when $x\neq0$ and $\sign x=0$.

Note that the left-hand side of the above inequality is $\| t^s v \|_{L^p([\del,2])}^p$. Dividing by $\|t^s v\|_{L^p ([\del,2])}^{p-1}$ (which is finite) from both sides and taking the limit as $\del \to 0$ we get \eqref{ldwtibp}. 

\medskip\noindent
\ref{Item::wtibple::2}
By assumption $u$ vanishes outside $\Bb^N$, so 
\begin{align*} 
\| \del ^{s} u \|^p_{L^p( \om) } 
&\lesssim  \int_{|y'| < 1} \int_{y_N= \rho (y') }^{1} (y_N - \rho (y') )^{sp} |u (y' , y_N)  |^p \, dy_N \,  dy'  
\\ &= \int_{|y'| < 1} \int_{t =0}^{1 - \rho(y')} t ^{sp} |u (y' , t + \rho  ( y') )|^p  \, dt \,  dy'. 
\end{align*}
Set $\wti{u} (y', t) := u  (y', t+ \rho(y'))$. Then $\wti{u} (y',t)$ vanishes near $t= 1-\rho(y' )$.  Since $\sup |\rho|<1$, for every $y'\in\R^{N-1}$, we see that the map $t\mapsto u(y',t+\rho(y'))$ is supported in $[0,1-\rho(y'))$ and vanishes near $1-\rho(y')$. Since $1- \rho(y') < 2$, by part \ref{Item::wtibple::1} we have
\begin{align*} 
\| \del ^{s} u  \|_{L^p(\om ) } 
&\lesssim\int_{|y'| < 1} \int_{0}^{1 - \rho(y')} t ^{(s+1)p} |D_t u  (y' , t + \rho  ( y') )  |^p \, dt \,  dy'  
\\ &= \int_{|y'| < 1} \int_{\rho (y')}^{1} (y_N - \rho(y'))^{(s+1)p} |D_{y_N} u  (y' , y_N )  |^p \, d y_N \,  dy'  
\\ &\lesssim \| \del^{s+1} D  u \|_{L^p(\om ) }. 
\end{align*}    
This completes the proof.
\end{proof}

\pr{H-L_int}   
	Let $1 < p < \infty$ and $k, l$ be non-negative integers with $l < k$. Let $ \Om$ be a bounded Lipschitz domain in $\R^N$ and define $\del (x)$ to be the distance function to the boundary $b \Om$. If $u \in W^{k,p}_{\loc} (\Om)$ and 
	\[
	\sum_{|\gm| \leq k}\| \del^{k- l} D^{\gm} u \|_{L^p (\Om)}  < \infty, 
	\]
	then $u \in W^{l, p} (\Om)$. Furthermore, there exists a constant $C$ that does not depend on $u$ such that
\[
	\| u \|_{W^{l,p} (\Om)}  \leq C
	\sum_{|\gm| \leq k}\| \del^{k- l} D^{\gm} u \|_{L^p (\Om)} . 
\]
\epr
\begin{proof}

	For each $0 \leq i \leq l < k$ and each $|\all| = i$, we show that
	\eq{ider}
	\int_{\Om}  |D^{\all}  u |^{p}  \, dV  (x) \lesssim \sum_{|\gamma|\le k}\int_{\Om}  \del(x)^{p (k-l)} | D^\gamma u |^{ p} \, dV (x) . 
	\eeq
It suffices to show that for every non-negative integer $j$ and $1 < p < \infty$, one has 
	\eq{wtibp}
	\| \del ^{ j} v \|_{L^p (\Om)}
	\lesssim_{j,p,\Omega}  \| \del^{j +1} v \|_{L^p(\Om)} +  \| \del^{j+1} Dv \|_{L^p(\Om)}.
	\eeq   
Indeed, setting $v =D^{\all} u$ and using \eqref{wtibp} $(k-i)$ times we get
\begin{align*}
    \| D^{\all} u \|_{L^p( \Om)}  &\lesssim \|\delta D^{\all} u \|_{L^p (\Om)} 
   + \| \del   D D^{\all} u  \|_{L^p(\Om)}
   \\
   &\lesssim\dots\lesssim\sum_{|\gamma|\le k-i}\|\delta^{k-i} D^{\all+\gamma} u \|_{L^p (\Om)} 
   \\
   &\lesssim\sum_{|\gamma|\le k}\|\delta^{k-l} D^{\gamma} u \|_{L^p (\Om)}.
\end{align*}
	
	It remains to prove \eqref{wtibp}. Take a finite open cover $\{\U_{\nu} \}_{\nu=0}^M$ of $\Om$ such that $\U_0 \subset \subset \Om$ and $\bigcup_{\nu=1}^M \U_\nu \supset b\Om$. Let $\{\chi_\nu \}$ be a partition of unity such that $\chi_\nu \in C^{\infty}_{c}(\U_{\nu})$, $0 \leq \chi_\nu \leq 1 $, and $\sum_{\nu=0}^M \chi_\nu \equiv 1 $ in some neighborhood of $\Om$. 
	We can assume that for each $1\le \nu\le M$ there exists an invertible affine linear transformation  $\psi_\nu: \B^N \to \U_{\nu}$, where $\B^N$ is the unit ball in $\R^N$, such that
\begin{gather*}
\psi_\nu( \B^N \cap \om_{\nu}) =  \mc{U}_{\nu} \cap \Om, \quad 1 \leq {\nu} \leq M. 
\end{gather*}  
Here $ \om_{\nu} = \{ y_N > \rho_{\nu} (y')  \}$ are special Lipschitz domains. 
For $y\in\B^N\cap\omega_\nu$, $ \del \circ\psi_{\nu}(y) \approx \del_{\nu}(y) := y_N -  \rho_{\nu}  ( y')$, thus
\begin{align*}
	\| \del^{j} v \|_{L^p (\Om)} 
&\lesssim \| \chi_0\del^{j} v \|_{L^p ( \Om)} 
	+   \sum_{\nu=1}^M \| \chi_\nu \del^{j} v  \|_{L^p (\Om \cap \mc{U}_\nu )}    
 \\ &\lesssim \| \chi_0 v \|_{L^p ( \Om)}  + \sum_{\nu=1}^M \| \del_\nu ^{j} [(\chi_\nu  v)\circ \psi_\nu] \|_{L^{p} ( \om_\nu\cap\B^N) }.  
\end{align*}
Clearly $\|\chi_0 v\|_{L^p(\Omega)}\lesssim\|\delta^{j+1}\chi_0v\|_{L^p(\Omega)}\le \|\delta^{j+1}v\|_{L^p(\Omega)}$.  By \rl{wtibple} \ref{Item::wtibple::2}, we have for $1\le \nu\le M$,
\begin{align*} 
\| \del_\nu ^{j} [(\chi_\nu  v)\circ \psi_\nu] \|_{L^{p} ( \B^N \cap \om_\nu) } 
	&\lesssim \| \del_\nu ^{j +1}  D [(\chi_\nu  v)\circ \psi_\nu]  \|_{L^{p} ( \B^N \cap \om_\nu)} 
	\\
	&\lesssim\| \del_\nu ^{j +1}  D (\chi_\nu  v)\circ \psi_\nu  \|_{L^{p} ( \B^N \cap \om_\nu)} \\
	&\lesssim \| \del^{j+1} D (\chi_\nu  v) \|_{L^p(\U_\nu\cap\Om)}
	\\
	&\lesssim\| \del^{j+1} v \|_{L^p(\U_\nu\cap\Om)} + \| \del^{j+1} D v \|_{L^p (\U_\nu\cap\Om)}\\
	&\lesssim\| \del^{j+1} v \|_{L^p(\Om)} + \| \del^{j+1} D v \|_{L^p (\Om)}.
\end{align*}
Taking sum over $0\le\nu\le M$, this proves \eqref{wtibp} and thus the proposition. 
\end{proof} 
In fact, we can also view \rp{H-L_int} as a weighted version of the Poincar\'e inequality.  
	\begin{lemma}\label{C3_int}  
		Let $\Om$ be a bounded Lipschitz domain. Denote $\delta(x)=\dist(x,b\Omega)$. 
		Then for any $k\in\N$ and $1<p<\infty$ there is a $C=C_{k,p,\Omega}>0$ such that 
		\[ 
		\| \delta^{-k}f\|_{L^p(\Omega)}\leq C\| f \|_{H^{k,p}_0(\Omega)}.
		\]
	\end{lemma}
	\begin{proof}
	We only need to prove the statement for special Lipschitz domain $\omega=\{(x', x_N) \in \R^N: x_N>\rho(x')\}$ and for $f\in H^{k,p}_0(\omega)$ which is supported in $\B^N$, namely,
	\eq{C3int_red}
		\| \delta^{-k}f\|_{L^p(\omega)}\leq C\| f \|_{H^{k,p}_0(\omega)},\quad\forall f\in H_0^{k,p}(\omega),\quad\supp f \seq \B^N.
	\eeq 
	
	 For a bounded Lipschitz domain one can use partition of unity and the result for special Lipschitz domains. We leave the reader to check 
	the details. 
	

	The case $k=0$ is trivial, so we assume $k>0$. Since $H_0^{k,p}(\omega)$ is the completion of $C_c^\infty(\omega)$ under the norm $H^{k,p}(\R^N)$ (see Definition \ref{Sob_dom_def} \ref{Item::Sob_dom_def::Domain}), it suffices to prove \re{C3int_red} for $f\in C_c^\infty(\omega)$ with uniform bounds. Indeed, for a general $f\in H^{k,p}(\om)$, take $(f_j)_{j=1}^\infty\subset C_c^\infty(\om)$ such that $\|f_j-f\|_{H^{k,p}_0(\omega)}\to0$, so then $(\delta^{-k}f_j)_{j=1}^\infty\subset L^p(\om)$ is a Cauchy sequence, and $\|\delta^{-k}f\|_{L^p(\om)}\le C\|f\|_{H^{k,p}_0(\om)}$ with the same constant.
	
	Since $\|\rho\|_{C^{0,1}}<1$, we know that $\frac12\delta(x) \leq |x_N-\rho(x')|\leq 2\delta(x)$ for all $x\in\R^N$, so we can replace $\delta(x)$ by $|x_N -\rho(x')|$.
	
	Let $g(t)\in C_c^\infty(0,2)$. By Taylor's theorem
\[  
	g(t)=\frac1{(k-1)!}\int_0^t g^{(k)}(s)(t-s)^{k-1}ds, \quad t>0.   
\] 
	Therefore
	\begin{align*}
	\|t^{-k}g(t)\|_{L^p_t(\R_+)}
	&\leq \frac1{(k-1)!}\Big\|\frac1t\int_0^t|g^{(k)}(s)|ds\Big\|_{L^p_t(\R_+)}
	\\&= \frac1{(k-1)!}\Big\|\int_0^2|g^{(k)}(\lambda t)|d\lambda\Big\|_{L^p(\R_+)}
	\\ &\leq \frac1{(k-1)!}\int_0^2\|g^{(k)}(\lambda \cdot)\|_{L^p(\R_+)}d\lambda
	\\&=\frac1{(k-1)!}\int_0^2\|g^{(k)}\|_{L^p(\R_+)}\lambda^{-\frac1p}d\lambda
	\\&=\frac{p/(p-1)}{(k-1)!}\|g^{(k)}\|_{L^p(\R_+)}.
	\end{align*}
	Now for each $x'\in\R^{N-1}$, set $g_{x'}(t):=f(x',t+\rho(x'))$ so $g_{x'}^{(k)}(t)=\partial_t^kf(x',t+\rho(x'))=(\partial_{x_N}^k f)(x',t+\rho(x'))$, we see that $\supp g \seq [0,2)$ since $\supp f \seq \B^N$. By Fubini theorem we have
	\begin{align*}
	 \int_\omega\left||x_N -\rho(x')|^{-k}f(x', x_N)\right|^p \, dV(x)
	&=\int_{\R^{N-1}}dx'\int_0^\infty |t^{-k}f(x',t+\rho(x'))|^pdt
	\\&=\int_{\R^{N-1}}\|t^{-k}g_{x'}(t)\|_{L^p_t(\R_+)}^pdx'
	\\ &\leq C_{k,p} \int_{\R^{N-1}}\|g_{x'}^{(k)}\|_{L^p(\R_+)}^pdx'
	\\ &= C_{k,p}  \int_{\R^{N-1}} \left( \int_{\R_+}|(\partial_{x_N}^kf)(x', t+\rho(x'))|^pdt \right) \, dx'
	\\&= C_{k,p} \int_\omega|\partial_{x_N}^kf(x)|^pdx \leq C_{k,p} \int_\omega|D^kf|^p \, dV(x). 
	\end{align*}
	Thus we have $\|\delta^{-k}f\|_{L^p(\omega)}\lesssim\|D^kf\|_{L^p(\omega)}\leq \|f\|_{W^{k,p}(\omega)}=\|f\|_{W^{k,p}(\R^N)}$ uniformly for all $f\in C_c^\infty(\omega)$.
	
	Note that by \rl{Lem::Hkp=Wkp},  $H^{k,p}(\R^N)=W^{k,p}(\R^N)$ with equivalent norm, and by density of $C_c^\infty(\omega)$ in $H^{k,p}_0(\omega)$ (see \rd{Sob_dom_def} \ref{Item::Sob_dom_def::Domain}) we conclude that $\|\delta^{-k}f\|_{L^p(\omega)}\lesssim\|f\|_{H^{k,p}(\R^N)}$ for all $f\in H_0^{k,p}(\omega)$. 
	\end{proof}

\pr{H-L_neg} 
	Let $\Om \seq\R^N$ be a bounded Lipschitz domain. Denote $\delta(x)=\dist(x, b \Om)$. Then for $s\geq 0$ and $1<p<\infty$, there is a $C=C(s,p,\Om)>0$ such that
	\begin{equation}\label{negindest}  
	\|u\|_{H^{-s,p}(\Om)}\leq C\|\delta^su\|_{L^p(\Om)},\quad\forall u\in L^p_{\loc}(\Om) .
	\end{equation} 
\epr

\begin{rem}
   For the special case when $p=2$, $s\ge0$ is not a half integer, and $\Omega$ has smooth boundary, the above result is proved in \cite[Theorem C.4]{C-S01}. 
\end{rem}

	\begin{proof}

	Note that the estimate is equivalent to showing the boundedness of the inclusion operator 
	\[
	\iota: L^p(\Om,\delta^{s}) \hra H^{-s,p}(\Om)
	\]
	for $s\ge0$ and $1<p<\infty$. We will argue by duality and interpolation.
	
	Let $p'$ be the conjugate of $p$. By \rp{C3_int}, for $k\in\N$,
	\[ 
	  \|g\|_{L^{p'}(\Omega,\delta^{-k})}\le C_{k,p'}\|g\|_{H_0^{k,p'}(\Omega)},\quad \forall g\in H_0^{k,p'}(\Omega),
	\]     
	is a bounded operator. By Proposition \ref{Prop::DualSpace} \ref{Item::DualSpace::H0Char}, we have $H^{-k,p}(\Om)=H_0^{k,p'}(\Om)'$. Using H\"older's inequality, we have for every $f\in L^p(\Om,\delta^{k})$,
	\begin{align*}
	\|f\|_{H^{-k,p}(\Om)}
	& = \sup_{\substack{g\in H^{k,p'}_0(\Om); \|g\|_{H^{k,p'}_0(\Om)}\leq 1}}\langle f,g \rangle 
	\\&\leq \sup_{\substack{g\in L^{p'}(\Om,\delta^{-k});\|g\|_{L^{p'}(\Om,\delta^{-k})}\leq C_{k,p'}}}\int_{\Om} |fg|
	\\ &= \sup_{\|\delta^{-k}g\|_{L^{p'}(\Om)}\leq C_{k,p'}} 
	\int_{\Om} |\delta^{k} f||\delta^{-k}g|
	\\& \leq \sup_{\|\delta^{-k}g\|_{L^{p'}(\Om)}\leq C_{k,p'}} \| \del^k f \|_{L^p (\Om)} \| \del^{-k} g \|_{L^{p'}(\Om)}
	\\&\leq  C_{k,p' } \| \del^k f \|_{L^p (\Om)}  
	=C_{k,p'} \|f \|_{L^p(\Om,\delta^{k})}. 
	\end{align*}
	Hence the inclusion $\iota:L^p(\Om,\delta^{k})\hookrightarrow H^{-k,p}(\Om)$ is bounded for $k \in \N$. 
	
	For general $s>0$, take any integer $k>s$ and denote $\theta=s/k$. By \rp{interpol_space} we have
	\[ 
	[L^p(\Om),L^p(\Om,\delta^{k})]_\theta=L^p(\Om,\delta^{s}), 
	\]
and  
\[ 
	[L^p(\Om),H^{-k,p}(\Om)]_\theta=H^{-s,p}(\Om).  
\]
Using interpolation we obtain the boundedness of inclusion $\iota: L^p(\Om,\delta^{s})\to H^{-s,p}(\Om)$. 
\end{proof}
	
	We now use \rp{H-L_neg} to extend \rl{C3_int} to all $s>0$.  
	
	\pr{C3}  
	Let $\Om$ be a bounded Lipschitz domain. Denote $\delta(x)=\dist(x,b\Omega)$. 
	Then for any $s \geq 0 $ and $1<p<\infty$ there is a $C=C_{s,p,\Omega}>0$ such that 
	$\| \delta^{-s}f\|_{L^p(\Omega)}\leq C  \|f\|_{H^{s,p}_0(\Om)}$.
	\epr 
	\begin{proof}
		Since the dual space of $L^p(\Om, \del^{-s})$ is $L^{p'}(\Om, \del^{s})$, we have
	\begin{align*}
	\| \del^{-s} f \|_{L^p (\Om)} 
	= \sup_{\substack{g \in L^{p'} (\Om, \del^{s}), \\ \| g \|_{L^{p'} (\Om, \del^{s})} \leq 1 } } \left|  \left<  f, g \right> \right| . 
	\end{align*}
	Since $L^{p'} (\Om, \del^{s}) \subset H^{-s,p'}(\Om)$ by \rp{H-L_neg} and since $H^{s,p}_0 (\Om) = H^{-s,p'} ( \Om )'$ from Proposition \ref{Prop::DualSpace} \ref{Item::DualSpace::H0Char}, we have
	\[
  	\| \del^{-s} f \|_{L^p (\Om)}  \leq	\sup_{\substack{g \in H^{-s,p'}(\Om), \\ \| g \|_{H^{-s,p' }(\Om) \leq C } } } \left|  \left<  f, g \right>  \right| 
	\approx\| f \|_{H^{s,p}_0 (\Om)}, 
	\]
	which completes the proof.
	\end{proof}
	
	\pr{H-L_lo} 
	Let $\Om \subseteq \R^N$ be a bounded domain with $C^2$ boundary. Then for $1 < p < \infty$ and $0 \le r \le 2$, there is a $C = C_{\Om, r, p}>0$ such that
	\begin{enumerate}[(i)]
	    \item\label{Item::H-L_lo::1} $ \| f \|_{H^{r,p} (\Om)} \leq C( \| \del^{1-r} f  \|_{L^p (\Om)} + \| \del^{1-r} Df \|_{L^p (\Om)})$ for $0\le r\le 1$ and $f\in W^{1,p}(\Omega,\delta^{1-r})$.
	    \item\label{Item::H-L_lo::2} $\| f  \|_{H^{r,p}(\Om)} \leq  C(\| \del^{2-r} f \|_{L^p(\Om)} + \| \del^{2-r} Df \|_{L^p(\Om)} + \| \del^{2-r} D^2 f \|_{L^p (\Om)})$ for $1\le r\le 2$  and $f\in W^{2,p}(\Omega,\delta^{2-r})$.
	\end{enumerate}
	\epr 
	
\begin{proof} 
			\ref{Item::H-L_lo::1} Since $b \Om$ is $C^2$ and $ 0 \leq s\leq 1$, we can apply \rp{Prop::EqvNormC2} \ref{Item::H-L_lo::1} to get 
			\[   
			\|f\|_{H^{s,p}(\Om)}\lesssim\|f\|_{H^{s-1,p}(\Om)}+\|Df\|_{H^{s-1,p}(\Om )}. 
			\]  
			Now $s-1\leq 0$ so \rp{H-L_neg} applies, and we have
			\begin{gather*}
			\|f\|_{H^{s-1,p}(\Om)}\lesssim\|\delta^{1-s}f\|_{L^p(\Om)} ; 
			\\ 
			\|Df\|_{H^{s-1,p}(\Om)}\lesssim\|\delta^{1-s}Df\|_{L^p(\Om )}. 
			\end{gather*}
			Combining them we get $\|f\|_{H^{s,p}(\Om)}\lesssim\|\delta^{1-s}f\|_{L^p(\Om)}+\|\delta^{1-s}Df\|_{L^p(\Om)}$, which proves \ref{Item::H-L_lo::1}. 
			
			\medskip\noindent
			\ref{Item::H-L_lo::2} Since $1\leq s\leq 2$ we have $s,s-1\in[0,2]$. So by \rp{Prop::EqvNormC2} \ref{Item::EqvNormC2::12}, 
			\begin{align*}
			\|f\|_{H^{s,p}(\Om)} 
			\lesssim \|f\|_{H^{s-2,p}(\Om )}+\|Df\|_{H^{s-2,p}(\Om)}+\|D^2f\|_{H^{s-2,p}(\Om )}.  
			\end{align*}
		Since $s-2\leq 0$, we again apply \eqref{H-L_neg} to get 
			\[
			\|D^jf\|_{H^{s-2,p}(\Om)}\lesssim\|\delta^{2-s}D^jf\|_{L^p(\Om)}, \quad j = 0, 1, 2. 
			\]
		Thus 
			\[
			\|f\|_{H^{s,p}(\Om)}\lesssim\|\delta^{2-s}f\|_{L^p(\Om)}+\|\delta^{2-s}Df\|_{L^p(\Om)}+\|\delta^{2-s}D^2f\|_{L^p(\Om )}, 
			\]
	which proves \ref{Item::H-L_lo::2}. 
	\end{proof}
	
	\section{Sobolev Estimates of Homotopy Operators}  
In this section we derive the weighted estimates for the homotopy operator. Together with the commutator estimate and Hardy-Littlewood lemma, this leads to the proof of \rt{thm1_intro}. Unlike in \cite{Gong19} and \cite{DG19}, no integration by parts is used in our proof.  

In what follows we let $\rho$ be a $C^2$ defining function of $\Om$ which is strictly plurisubharmonic in a neighborhood of $b\Om$. We will adopt the following notation:
\begin{gather*}
   \Om_{\ve} = \{ z \in \C^n: \dist(z, \Om) < \ve \} ,    
   \quad \Om_{-\ve} = \{ z \in b\Om: \dist(z, \Om) > \ve \}. 
\end{gather*}

	\pr{hf_C1} 
	Let $\Om$ be a bounded domain in $\C^n$ with $C^2$ boundary. Suppose $W (z, \zeta) \in C^1 (\Om_{\ve} \times (\Om_{\ve} \sm\overline \Om_{-\ve}))$ is a Leray mapping, that is $W$ is holomorphic in $z \in \Om_{\ve}$ and satisfies 
	\[ 
	\Phi(z, \zeta)  := W(z, \zeta) \cdot (\zeta -z) \neq 0,  \quad z \in \Om, \quad \zeta \in \Om_{\ve} \sm\overline \Om.
	\] 
	Let $\mc{U}$ be a bounded neighborhood of $ \ov{\Om}$ such that $U \subset \Om_{\ve}$. 
    Suppose $\var$ is a $(0,q)$-form with $1 \leq q \leq n$ such that $\var$ and $\dbar \var$ are in $C^1 (\ov{\Om})$. Then
	\[ 
	\var = \dbar \Hc_q \var + \Hc_{q+1} \dbar \var . 
	\] 
Here $\Hc_q$ is the operator defined by
\eq{hom_opt} 
  \Hc_q \var = \int_{\mc{U}} K^0_{0, q-1} \we E\var + \int_{\mc{U} \sm\overline \Om}  K^{01}_{0, q-1} \we [\dbar, E] \var,  
\eeq
	where $E$ is any extension operator that maps $C^{\infty}(\ov{\Om} ) $ to $C^{\infty} (\C^n)$ with $\supp E\var \seq \U$ for all $\var$, and  
	\[   
	K^{0} (z, \zeta) = \frac{1}{(2 \pi i)^{n}} \frac{\left<\ov{\zeta} - \ov{z} \, , \, d \zeta \right>}{|\zeta -z |^{2}} \we \left( \dbar_{\zeta,z} \frac{ \left< \ov{\zeta} - \ov{z} \, , \, d \zeta \right>}{|\zeta -z|^{2}} \right)^{n-1}, \quad \dbar_{\zeta,z} = \dbar_{\zeta} + \dbar_{z} ; 
	\] 
	\begin{align} \label{Om01}
	& K^{0, 1} (z, \zeta) 
	= \frac{1}{(2 \pi i)^{n}} \frac{\left<\ov{\zeta} - \ov{z} \, , \, d \zeta \right>}{|\zeta -z |^{2}} \we \frac{\left< W, d \zeta  \right>}{\left< W \, , \, \zeta -z \right>} 
	\\ \nonumber & \qquad \we \sum_{i+j=n-2} \left[ \frac{\left< d\ov{\zeta} - d\ov{z} \, , \, d \zeta \right>}{|\zeta -z |^{2}} \right]^{i}   \we \left[ \dbar_{\zeta,z} \frac{ \left< W  ,  d \zeta  \right>}{\left< W, \zeta -z \right>} \right]^{j}.
	\end{align}
	We set $K^{1}_{0,-1} = 0$ and $K^{0,1}_{0,-1} =0$.  
	
	\epr  
The reader can find the proof of \rp{hf_C1} in \cite{Gong19}. 
Here we note that on any bounded strictly pseudoconvex domain $\Om$ with $C^2$ boundary, there exists an $\ve>0$ such that $W$ satisfies the assumptions in \rp{hf_C1} on $\Om_{\ve} \times (\Om_{\ve} \sm \ov{\Om}_{-\ve})$. Furthermore, near every $\zeta^{\ast} \in b\Om $, one can find a neighborhood $\mc{V}$ of $\zeta^{\ast}$ such that for all $z \in \V$, there exists a coordinate map $\phi_{z}: \mc{V} \to \R^{2n}$ given by $\phi_{z}: \zeta \in \V \to (s, t) = (s_{1}, s_{2}, t_{3},\dots, t_{2n})$, where $s_1 = \rho(\zeta)$ (we have $s_1 \approx \del(\zeta)$ for $\zeta\in\V\sm\Omega$). 
Moreover for $z \in \V \cap \Om$, $\zeta \in \V \sm\overline \Om$, $\Phi  = W(z,\zeta) \cdot (\zeta -z)$ satisfies 
\eq{Phie1}
  |\Phi (z, \zeta) | \geq c \left( \del(z) + s_{1} + |s_{2}| + |t|^{2} \right),  \quad \del(z) = \dist(z, b\Om), 
\eeq
	\eq{Phie2}
	|\Phi(z, \zeta)| \geq c |z - \zeta|^{2}, \quad \quad |\zeta -z| \geq c( s_1 + |s_2| + |t|), 
	\eeq
	for some constant $c$ depending on the domain. The reader can refer to \cite{Gong19} for details.   

From now on we shall fix an open set $\U$ such that $ \Om \subset \subset \U \subset \subset \Om_{\ve} $, and we will use our extension operator $\Ec$ defined by formula \eqref{E_def} with $\supp \Ec\var \subset \U$ for all $\var$.

\begin{thm}\label{H_thm}  
  Let $\Omega\seq\C^n$ be a bounded strictly pseudoconvex domain with $C^2$ boundary. Given $1 \leq q \leq n$, let $\Hc_q$ be defined as in \eqref{hom_opt}, where the extension operator $\Ec$ is given by formula \eqref{E_def}. Then for any $1<p<\infty$ and $s>\frac1p $, $\Hc_q$ is a bounded linear operator $\Hc_q:H^{s,p}_{(0,q)}(\Omega)\to H^{s+\frac12,p}_{(0,q-1)}(\Omega)$.
\end{thm}
\begin{rem} 
   In fact, when $q=n$, then any extension of $\var$ is automatically $\dbar$ closed, so 
   \[ 
     \Hc_n \var = \int_{\U} K^0_{0,q-1} \we E \var. 
   \] 
  In this case for all $s \ge 0$ and $1 < p < \infty$, $\Hc_n:H^{s,p}_{(0,n)}(\Omega)\to H^{s+1,p}_{(0,n-1)}(\Omega)$ where $\Om$ is any bounded Lipschitz domain. See Proposition \ref{H0_prop}.
\end{rem}		

\rt{H_thm} allows us to prove a homotopy formula under much weaker regularity assumption. 
		
\begin{thm}[Homotopy formula] 
Let $\Omega\seq\C^n$ be a bounded strictly pseudoconvex domain with $C^2$ boundary. Given $1<p<\infty$ and $1 \leq q \leq n$, 
suppose $\varphi \in H^{s,p}_{(0,q)}(\Omega)$ satisfies $\dbar\varphi\in H^{s,p}_{(0,q+1)}(\Omega)$ where $s> \frac1p$. 
Let $\Hc_q$ be defined by \eqref{hom_opt}, where the extension operator $E=\Ec_\Omega$ is given by formula \eqref{E_def}. Then the following homotopy formula holds in the sense of distributions: 
	\begin{equation}\label{hf_dist} 
	\varphi=\dbar \Hc_q\varphi+\Hc_{q+1}\dbar\varphi.
	\end{equation}
In particular for a $\dbar$-closed $\var$ which is in $H^{s,p}_{(0,q)}(\Om)$ for $s > \frac{1}{p}$, the equation $\dbar \Hc_q \var =\varphi$ holds and $\Hc_q \var \in H^{s+\yh, p}_{(0,q-1)} (\Om)$. 
\end{thm}
\begin{proof}

Formula \re{hf_dist} is proved in \cite{Gong19} for Stein's extension operator for $\var, \dbar \var \in C^1(\ov{\Om} )$. 
The statement holds for smooth forms by \rp{hf_C1}, and we shall use approximation for the general case. 
First we show that there exists a sequence $\var_{\ve} \in C^{\infty} (\ov{\Om} )$ such that 
	\begin{gather*}
	\var_{\ve} \overset{\ve \to 0}{\ra} \var \quad \text{in $H^{s,p}(\Om)$},   
	\\ 
	\dbar \var_{\ve} \overset{\ve \to 0}{\ra}  \dbar \var \quad \text{in $H^{s,p}(\Om)$}. 
	\end{gather*}
The smoothing is done componentwise, and for simplicity we will continue to denote the coefficient functions of $\var$ by $\var$. Take an open covering $\{ U_{\nu} \}_{\nu=0}^M$ of $\Om$ such that 
\[
	U_0  \subset\subset \Om, \quad b \Om \seq \bigcup_{ \nu=1}^M U_{\nu},  \quad U_{\nu} \cap \Om = \Phi_\nu(\{  x_N > \rho_{\nu} (x') \}),  \quad \nu=1, \dots, M. 
\]
Here $\Phi_\nu$, $1\le \nu\le M$ are some invertible affine linear transformations.

Let $\chi_{\nu}$ be a partition of unity associated with $\{U_{\nu} \}$, i.e. $\chi_{\nu} \in C^{\infty}_c(U_{\nu}) $ and $\sum \chi_{\nu} = 1$.  
	
Let $\Bb^{2n}$ be the unit ball in $\C^n$. For each $1 \leq \nu \leq M$, we can find an open cone $K_{\nu}$ and some $\ve_{\nu}$ such that $(U_{\nu} \cap \Om) + (K_{\nu} \cap \Bb^{n} ( \ve_{\nu})) \seq \Om$.
	
Take $\psi_0 \in C^{\infty}_{c} (\B^N)$ with $\psi_0 \geq 0$ and $\int_{\C^{n}} \psi_0 = 1$. For $1 \leq \nu \leq N$, take $\psi_{\nu} \in C^{\infty}_{c} (-K_{\nu})$ with $\psi_{\nu} \geq 0$ and $\int_{\C^{n}} \psi_\nu = 1$. 
Write $\psi_{\nu, \ve} (x) = \ve^{-2n} \psi _{\nu} (\frac{x}{\ve})$. 
For $\ve>0$ sufficiently small, we can define 
	\begin{gather*}   
	(\chi_0 \var) \ast \psi_{0, \ve} (z) 
	= \int_{\B^N} ( \chi_0 \var) (z - \ve \zeta) \psi_0 (\zeta) \, d V(\zeta), \quad z \in U_0, 
	\\
	(\chi_{\nu} \var) \ast \psi_{\nu, \ve} (z) 
	=\int_{-K} (\chi_j \var) (z - \ve \zeta) \psi_{\nu} (\zeta) \, d V(\zeta), \quad z \in U_{\nu} \cap \Om, \quad \nu = 1, \dots, M. 
	\end{gather*}  
	Setting $\var_{\ve} := \sum_{\nu=0}^M (\chi_{\nu} \var) \ast \psi_{\nu, \ve} \in C^{\infty} (\ov{\Om})$. Clearly 
	$\| \var_{\ve} - \var \|_{H^{s,p}(\Om)} \to 0$ since $	\| ( \chi_{\nu} \var) \ast \psi_{\nu, \ve} - \chi_{\nu} \var \|_{H^{s,p} (\Om)}   \overset{\ve \to 0}{\ra}  0$ for each $0 \leq \nu \leq M$.  Also, 
	\[
	\dbar \var_{\ve} = \sum_{\nu =0}^M \psi_{\nu, \ve} \ast \left( \dbar \chi_{\nu} \var + \chi_{\nu} \dbar \var \right).
	\]
	By assumption, both $\dbar \chi_{\nu} \var $ and $\chi_{\nu} \dbar \var$ are in $H^{s,p} (\Om)$, so $\psi_{\nu, \ve} \ast \left( \dbar \chi_{\nu} \var + \chi_{\nu} \dbar \var \right)$ converges to $\chi_{\nu} \dbar \var + \dbar \chi_{\nu} \var $ in $H^{s,p} (\Om)$. Taking the sum over $\nu$ we see that $\| \dbar \var_{\ve} - \dbar \var \|_{H^{s,p} (\Om)} \to 0$. 
	
Now \eqref{hf_dist} holds for $\var$ replaced with $\var_{\ve}$. By \rt{H_thm},  
	\begin{align*}
	\| \dbar \Hc_q (\var_{\ve} - \var ) \|_{H^{s-\yh,p} (\Om)} 
	&\leq  \| \Hc_q (\var_{\ve} - \var ) \|_{H^{s +\yh,p} (\Om)} 
	\\& \leq  \|  \var_{\ve} - \var \|_{H^{s,p} (\Om)},   
	\end{align*}
	and also 
	\begin{align*}
	\| \Hc_{q+1} \dbar (\var_{\ve} - \var) \|_{H^{s + \yh,p} (\Om)}
	&\leq \| \dbar (\var_{\ve} - \var) \|_{H^{s,p} (\Om)}. 
	\end{align*}
Then \eqref{hf_dist} follows by taking $\ve \to 0$. 

Note that if $\varphi\in H^{s,p}$ is $\dbar$-closed, then the distribution $\dbar\varphi( =0)$ is in $H^{s,p}(\Omega)$. Therefore \eqref{hf_dist} holds for this $\varphi$. In particular $u=\Hc_q\varphi$ is an $H^{s+\frac12,p}_{(0,q-1 )}$ form that solves $\dbar u=\varphi$.
\end{proof} 
		
First we prove a lemma which will be useful later. 
	
	\begin{lemma}\label{intestle} 
	Let $ n \geq 2$, $\beta \geq 0$, $\all>-1$, and let $0 < \del < \yh $. If $\all < \beta - \yh $, then 
		\[
		\int_0^1 \int_0^1 \int_0^1 \frac{s_1^{\all} t^{2n-3} \, ds_1 \, ds_2 \, dt}{(\del + s_1 + s_2 + t^2) ^{2+\beta}  (\del + s_1+ s_2 + t)^{2n-3} } \leq C  \del^{\all - \beta + \yh}.  
		\] 
	\end{lemma}
\begin{proof}

	Partition the domain of integration into seven regions: \\ 
	$R_{1}: t> t^{2} > \del, s_{1}, s_{2}$. We have
	\[
	I \leq \int_{\sqrt{\del}}^{1} \frac{t^{2n-3}}{t^{4 + 2\beta} t^{2n-3}} \left( \int_{0}^{t^{2}} s_{1}^{\all} \, ds_{1} \right) \left( \int_{0}^{t^{2}} \, ds_{2} \right) \, dt 
	\leq C \int_{\sqrt{\del}}^{1} t^{2 \all - 2\beta } \, dt \leq C \del^{\all - \beta + \yh}. 
	\] 
	$R_{2}: t> \del > t^{2}, s_{1}, s_{2}$. We have
	\[ 
	I \leq  \del^{-2 - \beta} \left( \int_{\del}^{\sqrt{\del}} \frac{t^{2n-3} }{t^{2n-3}} \, dt \right) \left( \int_{0}^{\del} s_{1}^{\all} \, ds_{1} \right) \left( \int_{0}^{\del} \, ds_{2} \right)
	\leq C \del^{\all - \beta + \yh}. 
	\] 
	$R_{3}: t> s_{1} > \del, t^{2}, s_{2}$. We have 
	\[
	I \leq \int_{\del}^{1} \frac{s_{1}^{\all }}{s_{1}^{2 + \beta}} \left( \int_{0}^{\sqrt{s_{1}}} \frac{t^{2n-3}}{t^{2n-3}} \, dt \right) \left( \int_{0}^{s_{1}} \, ds_{2} \right) \, ds_{1} 
	\leq C \int_{\del}^1 s_1^{\all - \beta + \yh - 1} \, ds_1 
	\leq C \del^{\all- \beta + \yh}. 
	\]
	$R_{4}: t> s_{2} > \del, t^{2}, s_{1}$. We have
	\[ 
	I \leq \int_{\del}^{1} \frac{1}{s_{2}^{2+\beta} } \left( \int_{0}^{\sqrt{s_{2}}} \frac{t^{2n-3}}{t^{2n-3}} \, dt \right) 
	\left( \int_{0}^{s_{2}} s_{1}^{\all} \, ds_{1} \right) \, ds_{2}   
	\leq C \int_{\del}^1 s_2^{\all - \beta + \yh - 1} \, ds_2  
	\leq C \del^{\all - \beta + \yh}.  
	\]
	$R_{5}: \del > t, t^{2}, s_{1}, s_{2}$. We have
	\[ 
	I \leq \del^{-2 -\beta } \del^{-(2n-3)} \left( \int_{0}^{\del} t^{2n-3} \, dt \right) \left( \int_{0}^{\del} s_{1}^{\all} \, ds_{1} \right)    \left( \int_{0}^{\del} \, ds_{2} \right) 
	\leq C \del^{\all - \beta +1}. 
	\] 
	$R_{6}: s_{1} > \del, t, t^{2}, s_{2}$. We have
	\[ 
	I \leq \int_{\del}^{1} \frac{s_{1}^{\all }}{s_{1}^{2 + \beta} s_{1}^{2n-3}} \left( \int_{0}^{s_{1}} t^{2n-3} \, dt \right) \left( \int_{0}^{s_{1}} \, ds_{2} \right) \, ds_{1} 
	\leq C  \int_{\del}^1 s_1^{\all - \beta} \, ds_1. 
	\] 
	$R_{7}: s_{2} > \del, t, t^{2}, s_{1}$. We have
	\[
	I \leq \int_{\del}^{1} \frac{1}{s_2^{2+\beta} s_2^{2n-3}} \left( \int_{0}^{s_{2}} t^{2n-3} \, dt \right) \left( \int_{0}^{s_{2}} s_{1}^{\all } \, ds_{1} \right)  \, ds_{2}  
	\leq C  \int_{\del}^1 s_2^{\all - \beta} \, ds_2. 
	\]
	Here the constants depend only on $n$, $\all$ and $\beta$. For $R_6$ and $R_7$, we have  
	\[   
	\int_\delta^1r^{\alpha-\beta}dr\le \begin{cases}C, &\alpha-\beta >-1, \\C (1+ |\log\del| ), &\alpha-\beta= -1, \\
	C\delta^{\alpha-\beta + 1}, &\alpha-\beta< -1,\end{cases} 
	\] 
	which is bounded by $ C\delta^{\alpha-\beta+\frac12}$ in all cases. 
\end{proof}	

We now write the homotopy operator $\Hc_q \var$ as 
\begin{align}  \label{H0H1}
 \Hc_q \var 
 &= \Hc_q^0 \var + \Hc_q^1 \var, 
\end{align}
where  
\[
  \Hc_q^0 \var: = \int_{\Uc} K^0_{0, q-1} \we \Ec\var, \quad 
  \Hc_q^1 \var : = \int_{\Uc} K^{01}_{0, q-1} \we [\dbar, \Ec] \var. 
\]

For the operator $\Hc_q^0$, we can gain one derivative for any $\var \in H^{s,p}(\Om)$, $s \geq 0$. 

\begin{prop} \label{H0_prop} 
  Let $1< p < \infty$ and $s \geq 0$. Suppose $\var \in H^{s,p}_{(0,q)}(\Om)$ with $ q \geq 1$. Then $\Hc_q^0 \var$ is in $H^{s+1,p}_{(0,q-1)} (\Uc)$, and 
  \[
     \| \Hc_q^0 \var \|_{H^{s+1,p} (\U)} \lesssim \| \Ec\var \|_{H^{s,p} (\U)} \lesssim \| \var \|_{H^{s,p}(\Om)}. 
  \]                                                
\end{prop}
\begin{proof}
The proof for integers $s$ can be found in \cite[Proposition 3.2]{DG19}. The general case follows from interpolation (see Propositions \ref{cinterpol} and \ref{interpol_space}). 
\end{proof}

\pr{mderest} 
Let $\Om \seq \C^{n}$ be a bounded strictly pseudoconvex domain with $C^{2}$ boundary, and $1 < p < \infty$. For $q \geq 1$, let $\Hc_q^1 \var$ be given by  \eqref{H0H1}, where the extension operator $E$ is defined by formula \eqref{E_def}. 
Suppose $\var \in H^{s,p}_{(0,q)}(\Om)$ for $s > \frac{1}{p}$ and $m$ is a positive integer such that $m-s-\frac12>0$. Then there exists some constant $C = C(\Om, p)$ such that
\[ 
   \| \del^{m-s-\yh} D^m \Hc_q^1 \var \|_{L^p(\Om)}  < C \| \var \|_{H^{s,p}  (\Om)}.  
\]  
\epr
\begin{proof}
Note that in view of \rrem{commest_rem} \ref{commest_rem3}, $\Hc^1_q \var$ is $C^{\infty}$ in the interior of $\Om$. For the estimate we will first show that 
\eq{west}  
    \| \del^{m-s-\yh} D^m \Hc_q^1 \var \|_{L^p(\Om)} \leq C \| \del^{1-s} [\dbar, \Ec] \var \|_{L^p(\U \sm\overline \Om) }, 
\eeq 
and then use \rt{commest_thm}. One needs to be careful however since \rt{commest_thm} applies only to special Lipschitz domains. We now explain how one can get around this issue. 
Recall that the extension operator $\Ec$ on a locally Lipschitz domain $\Om$ is defined by (see \eqref{E_def}) 
\[
   \Ec f =  \chi_0 f + \sum_{j=1}^M \chi_j E_{j} (\chi_j f). 
\] 
Thus
\begin{align*}
[D, \Ec] f 
&= D (\chi_0 f ) - \chi_0 (Df ) + \sum_{j=1}^M D ( \chi_j \cdot E_j (\chi_j  f))  - \chi_j \cdot E_j (\chi_j \cdot (D f))  
\\ &= (D \chi_0) f + \sum_{j=1}^M \chi_j [D, E_j] (\chi_j f) + (D \chi_j) E_j (\chi_j f ) + \chi_j \cdot E_j ((D \chi_j)  f) . 
\end{align*} 
Now we would like to bound $\| \del^{1-s} [D,\Ec]f \|_{L^p (\U \sm\overline \Om)}$ by $\| f \|_{H^{s,p} (\Om)}$. 
We need to estimate
\eq{comm_red}  
\sum_{j=1}^M \| \del^{1-s} [D, E_j] ( \chi_j f) \|_{_{L^p(U_j \cap \overline\Omega^c) } } 
+ \| \del^{1-s} g \|_{L^p(U \sm\overline \Om) }, 
\eeq 
where 
\[
  g:= (D \chi_0) f + \sum_{j=1}^M  (D \chi_j) E_j (\chi_j f) + \chi_j \cdot E_j ((D \chi_j) f), \quad \text{on $\C^n$. } 
\] 
The first term in \eqref{comm_red} is bounded by $C \| f \|_{H^{s,p} (\Om)} $ by  \rl{Enu_commest_le}. 
For the second term, notice that since $E_j $ is an extension operator on $U_j \cap \Om $, one has $E_j (\chi_j f) = \chi_j f$ and $E_j ((D \chi_j) f) = (D \chi_j) f $ in $U_j \cap \Om$. 
Hence by the condition $\chi_0 + \sum_j \chi_j^2 = 1$ we have $g \equiv 0$ in $\Om$, in other words, $\supp g \subseteq \U \sm \Om$. In view of \rp{Prop::DualSpace} \ref{H_0space}, we have $g \in H^{s,p}_0 (\U \sm\overline{\Om})$. Clearly $\| g \|_{H^{s,p}_0(\U \sm\overline{\Om})} \leq C\| f \|_{H^{s,p}(\Om)}$ since for each $j$, $E_j$ is an extension operator on $U_j \cap \Om$.  
By \rp{C3}, the following holds for any $s \geq 0$: 
\begin{align*}
\| \del^{1-s} g \|_{L^p(\U \sm\overline \Om) }  \lesssim \| \del^{-s} g \|_{L^p(\U \sm\overline \Om) }\lesssim \| g \|_{H^{s,p}_0(\U \sm\overline \Om)}\lesssim \|  f \|_{H^{s,p}(\Om)}. 
\end{align*}

We now proceed with the proof of \eqref{west}, for which we will estimate 
\eq{k=1int}
\int_{\Om} \del(z)^{p (m - s - \yh )} \left|  D_{z}^{m} \int_{\U \sm\overline \Om} K^{01}_{0,q} (z, \zeta) \we [\dbar, \Ec] \var (\zeta) \, dV(\zeta)   \right|^{p} \, dV(z), 
\eeq  
where in the definition of $K^{01}_{0,q}$ (see \eqref{Om01}) we set $W$ to be a $C^1$ Leray map. 
Writing $ \Phi(z, \zeta) = W(z, \zeta) \cdot (\zeta -z)$, the inner integral can be expanded to a linear combination of 
\begin{gather} \label{k1f} 
\mc{K} f (z) =  \int_{\mc{U} \sm\overline \Om}  f( \zeta) P (W_1 (z, \zeta), z, \zeta) \frac{N_{1} (\zeta-z) }{\Phi^{n-l} (z, \zeta) |\zeta -z|^{2l}} \, dV(\zeta) , \quad  1 \leq l \leq n-1. 
\\  \nn
W_1 = (W,  \pa_{\zeta} W, \pa_z^{k_0} W  ),  \: k_0 \leq m.    
\end{gather} 
Here $f$ is a coefficient function of $[\dbar, \Ec] \var$. $P(w)$ denotes a polynomial in $w$ and $\ov{w}$, and $N_1$ denotes a monomial of degree $1$ in $\zeta - z$ and $\ov{\zeta -z}$. $P$ may differ when  recurs. 

By the remark after \rp{hf_C1}, we can take a small neighborhood $\V$ of a fixed boundary point $\zeta^{\ast} \in b \Om$. For $z \in \V$, let $\phi_{z}: \V \to \phi(\V)$ be the coordinate transformation satisfying \eqref{Phie1} and \eqref{Phie2}.
Using a partition of unity in $\zeta$ space and replacing $f$ by $\chi f$ for a $C^{\infty}$ cut-off function $\chi$, we may assume
\[
\text{supp}_{\zeta} \, f \seq \V \sm\overline \Om. 
\]
Similarly by a partition of unity in $z$ space and replacing $K^{01}_{0,q}$ by $\chi K^{01}_{0,q}$ we may assume 
\[ 
\supp_{z} \, K^{01}_{0,q} (z, \zeta) \seq \V \cap \Om. 
\] 
Write $\wti{N}_{1-2l}(z,\zeta) = \frac{N_1 (\zeta-z ) }{|\zeta -z |^{2l}}$. 
For $z \in \V \cap \Om$ and $\zeta \in \V \sm\overline \Om$, we have
\eq{Ntizde}
| \pa_{z} ^{j} \wti{N}_{1-2l} ( \zeta -z) | \lesssim |\zeta - z|^{1-2l - j}, 
\eeq
\eq{Phizde}
|\pa_{z}^{j} \Phi^{-(n-l)} (z, \zeta)| \lesssim | \Phi^{-(n-l) - j} (z, \zeta) | . 
\eeq
Here we use the fact that $W$ is holomorphic in $z \in \U$.  
Write
\[ 
 D_{z}^{m} \mc{K} f (z) = \int_{\U \sm\overline \Om} A_m(z, \zeta) f(\zeta) \, dV(\zeta), 
\]
where  $A_m (z, \zeta)$ is a sum of terms of the form
\begin{equation} \label{Am}
A_m (z, \zeta) = \frac{P_{1} (z, \zeta)}{\Phi^{n-l + \mu_1} (z, \zeta)} \pa_{z}^{\mu_2} \wti{N}_{1-2l} (z,\zeta), \quad 1 \leq l \leq n-1, \quad \mu_1+\mu_2 \leq m. 
\end{equation}
Setting $\frac{1}{p} + \frac{1}{p'} =1$, we have
\begin{align*}
| D_{z}^{m} \mc{K} f (z) | &\leq \int_{\U \sm\overline \Om} |A_m(z, \zeta)|^{\frac{1}{p}}  |A_m(z, \zeta)|^{\frac{1}{p'}} |f(\zeta)| \, dV(\zeta)  \\
&= \int_{\U \sm\overline \Om} \del(\zeta)^{-\eta} |\del (\zeta)^{\eta} A_m(z, \zeta)|^{\frac{1}{p}}  |  \del (\zeta)^{\eta} A_m(z, \zeta)|^{\frac{1}{p'}} |f(\zeta)| \, dV(\zeta), 
\end{align*}
where $\eta$ is a number to be specified. 
By Hölder's inequality, we get
\eq{secdrfppe}
| D_{z}^{m} \mc{K} f (z)|^{p} \leq \left[ \int_{\U \sm\overline \Om} \del (\zeta)^{- \eta p + \eta} |A_m(z, \zeta)| |f(\zeta)|^{p} \, dV(\zeta)  \right] \left[ \int_{\U \sm\overline \Om} \del( \zeta)^{\eta} |A_m (z, \zeta)| \, dV(\zeta)  \right]^{\frac{p}{p'}}. 
\eeq
By \eqref{Phie2}, we have $C' | \zeta -z | \geq | \Phi (z, \zeta) | \geq  C| \zeta -z |^{2}$. In view of \eqref{Ntizde} and \eqref{Phizde}, it suffices to estimate $A_m (z, \zeta) $ for $l=n-1$, $\mu_1 = m$ and $\mu_2 =0 $. Thus from now on we can just assume 
\[ 
A_m(z, \zeta) = \frac{P(W_1,z,\zeta)}{\Phi^{m+1} (z,\zeta)}  \wti{N}_{-(2n-3)}. 
\]
By estimate \eqref{Phie1}, we have for $z \in \V \cap \Om$ and $\zeta \in \V \sm\overline \Om $, 
\eq{Philbe}
|\Phi (z, \zeta) | \geq c (\del(z) + |s_{1}| + |s_{2} | + |t| ^{2} ), \quad |\zeta - z| \geq c(\del(z)+ |s_1| + |s_2| + t), 
\eeq
where $(s_{1}, s_{2}, t) = \left( \phi^{1}_{z}(\zeta) = \rho(\zeta), \phi^{2}_{z} (\zeta), \phi'_{z}(\zeta ) \right) $.

By \eqref{Philbe} and integrating in polar coordinates $t = (t_{1},\dots,t_{2n-2}) \in \R^{2n-2}$, we obtain
\begin{align} \label{Amint1} 
&\int_{\U \sm\overline \Om} \del(\zeta)^{\eta} |A_m(z, \zeta)| \, dV(\zeta)\\  
&\leq C_{0}   \int_{s_1 =0}^{1} \int_{s_2=0}^1\int_{t =0}^{1} \frac{s_1^{\eta} t^{2n-3}\, ds_1 \, ds_2  \, dt}{(\del(z) + s_1 + s_2 + t^{2})^{2+(m-1)}{( \del(z) + s_1 + s_2 + t) }^{2n-3} } \\ \nn
&\leq C_{0} \del(z)^{\eta + \yh - (m-1)} = C_0 \del(z)^{\eta - m + \frac{3}{2} }, 
\end{align}
where we apply \rl{intestle} using $\all = \eta$, $\beta = m -1 \geq 0$ and by choosing 
\eq{etarange1}
   -1 < \eta < \beta - \yh =  m- \frac{3}{2} .  
\eeq
 Note that the constant $C_{0}$ depends only on the domain $\Om$ and the defining function $\rho$. Hence by \eqref{secdrfppe} and applying Fubini theorem we get
\begin{align} \label{Hqfdsetup}
& \int_{\Om} \del(z)^{p (m- s - \yh ) } |D_{z}^{m} \mc{K}  f(z)|^{p} \, dV(z)  
\\ \nonumber & \qquad \lesssim 
\int_{\Om} \del(z)^{\gm}\left( \int_{\U \sm\overline \Om} \del(\zeta)^{(1-p) \eta } |A_m(z, \zeta) | |f(\zeta)|^{p} \, dV(\zeta) \right) \, dV(z)
\\ \nonumber & \qquad \lesssim \int_{\U \sm\overline \Om} \del(\zeta)^{(1-p) \eta } \left[ \int_{\Om}  \del(z)^{\gm} |A_m(z, \zeta)| \, dV(z)   \right] |f(\zeta)|^{p} \, dV(\zeta), \quad 
\end{align} 
where
\eq{mtfdgm'}
\gm = p \left( m - s - \yh \right) + \left( \eta -m + \frac{3}{2} \right) \frac{p}{p'}  = m - \frac{3}{2} + (1-s)p + \eta (p-1). 
\eeq

We now estimate the inner integral in the last line of \eqref{Hqfdsetup}.  
Recall that for each $z \in \V$, we define $C^{1}$ coordinate transformation $\phi_{z}$ for $\zeta \in \V$: 
\[ 
\phi^{1}_{z} (\zeta) = \rho (\zeta), \quad \phi^{2}_{z} (\zeta) = \IM (\rho_{\zeta} \cdot (\zeta -z)), \quad \phi'_{z} (\zeta) = \left( \RE(\zeta' -z'), \IM(\zeta'- z') \right).  
\] 
Now for $\zeta \in \V$, we can similarly define a new coordinate $\wti{\phi}_{\zeta}: \V \to \C^n$ for $z \in \V$:
\begin{gather*} 
\wti{\phi}^{1}_{\zeta} (z) = \rho (z), \quad \wti{\phi}^{2}_{\zeta}(z) = \IM (\rho_{\zeta} \cdot (\zeta -z)),  \quad \wti{\phi}'_{\zeta} (z) = \left( \RE(\zeta' -z'), \IM(\zeta'- z') \right).
\end{gather*}
Write $(\ti{s}_{1}, \ti{s}_{2}, \ti{t}) = (\wti{\phi}^{1}_{\zeta} (z), \wti{\phi}^{2}_{\zeta} (z), \wti{\phi}'_{\zeta} (z)) $ where $|\wti{\phi}^{1}_{\zeta}(z)| =|\rho(z)| \approx \del(z)$.   
By \eqref{Philbe} we have for $z \in \V \cap \Om$ and $\zeta \in \V \sm\overline \Om$, 
\begin{align} \label{Phiest}
|\Phi (z, \zeta) | &\geq c (\del(z) + \phi^{1}_{z} (\zeta) + | \phi^{2}_{z}(\zeta) | + |\phi'_{z} (\zeta)|^{2} ) \\ \nonumber
&\geq c (\del(\zeta) + |\wti{\phi}^{1}_{\zeta} (z)| + | \wti{\phi}^{2}_{\zeta}(z) | + |  \wti {\phi}' _{\zeta} (z)|^{2} ) \\ \nonumber
&= c (\del(\zeta) + |\ti{s}_{1}| + |\ti{s}_{2}| + |\ti{t} |^{2} ), 
\end{align}
and 
\eq{}
|\zeta - z| \geq c (|\del(\zeta) + |\ti s_1| + |\ti s_2| + |\ti t|)
\eeq
Writing in polar coordinates $\ti{t}=(t_1,\dots, t_{2n-2}) \in \R^{2n-2}$, we have 
\begin{align} \label{Hqfdest}
\int_{\Om}  \del(z)^{\gm} |A_m(z, \zeta)| \, dV(z) 
\leq C \int_{\ti{s}_1 =0}^{1} \int_{\ti{s}_2 =0}^{1} \int_{\ti{t} =0}^{1} \frac{\ti{s}_1^{\gm} \,  \wti{t}^{2n-3} \, d \ti{s}_1 \, d \ti{s}_2  \, d \ti{t}}{( \del(\zeta) + \ti{s}_1 + \ti{s}_2 + \ti{t}^{2})^{m+1} \, (\del(\zeta)+ \ti{s}_1 + \ti{s}_2 + \ti{t})^{2n-3}}. 
\end{align} 
To apply \rl{intestle} we take $\all = \gm$ and $\beta = m-1$, and we need $ -1 < \gm < \beta - \yh = m - \frac{3}{2}$. In view of \eqref{mtfdgm'}, this is the same as
\[  
   -1  <  m - \frac{3}{2} + (1-s)p + \eta (p-1) < m - \frac{3}{2}, 
\] 
which translates to 
\eq{etarange2}
\frac{(s-1)p}{p-1} - \frac{m-\yh}{p-1}
< \eta < \frac{(s-1)p}{p-1} = (s-1) p'. 
\eeq 
Note that this is always possible by our assumption on $m$.  
Now in view of both \eqref{etarange1} and \eqref{etarange2} we need to choose 
\eq{etarange} 
\max \left\{ -1, \frac{(s-1)p}{p-1} - \frac{m-\yh}{p-1} \right\} < \eta < \min \left\{ m-\frac{3}{2}, \frac{(s-1) p}{p-1} \right\}, \quad  \frac{1}{p} + \frac{1}{p'} = 1. 
\eeq  
By assumption $s>\frac1p$ and $m>s+\frac12$, so the range of admissible $\eta$ is non-empty.

Now applying \rl{intestle} we obtain 
\[
   \int_{\Om}  \del(z)^{\gm} |A_m(z, \zeta)| \, dV(z)
   \leq C \del(\zeta)^{\gm + \yh - (m-1) }  = C \del(\zeta)^{\gm - m + \frac{3}{2}},             
\]
where the constant depends on $\Om$ only. 

Using \eqref{Hqfdest} in \eqref{Hqfdsetup} we get 
\begin{align*}
\left[ \int_{\Om} \del(z)^{ \left( m-s - \yh \right) p} |D_{z}^{m} Kf(z)|^{p} \, dV(z)   \right]^{\frac{1}{p}}
&\lesssim \left[ \int_{\U \sm\overline \Om} \del(\zeta)^{ (1- p) \eta } \del (\zeta)^{\gm - m + \frac{3}{2}} |f (\zeta) |^{p} \, dV(\zeta) \right]^{\frac{1}{p}} \\
&\lesssim \left[ \int_{\U \sm\overline \Om} \del(\zeta)^{(1-s) p } | [\dbar, \Ec] \var (\zeta )|^p \, dV(\zeta) \right]^{\frac{1}{p}}. 
\end{align*}
This proves \eqref{west} and thus the proposition. 
\end{proof}

Next we extend the result of \rp{mderest} to all lower order derivatives of $u$. 

\pr{loder_est}
Keeping the assumptions of \rp{mderest}, the following holds
\[
   \| D^{k} \Hc_q^1 \var \|_{L^p (\Om, \, \del(z)^{m-s -\yh}) } \leq C (\Om, p) \| \var \|_{H^{s,p}(\Om)}, 
\quad 0 \leq k \leq m. 
\]
\epr 
\begin{proof}

We need to estimate 
\[
\int_{\Om} \del(z)^{p (m - s - \yh )} \left|  D_{z}^{k} \mc{K} f \right|^p \, dV(z), 
\]
where 
\[
  \mc{K}f (z)= \int_{\U \sm\overline \Om} K^{01}_{0,q} (z, \zeta) \we [\dbar, \Ec] \var (\zeta) \, dV(\zeta)  , 
  \quad f = [\dbar, \Ec] \var. 
\]
As before we write 
\[
D_z^k \mc{K} f (z) = \int_{\U \sm\overline \Om} A_k (z, \zeta) f (\zeta) \, dV (\zeta), 
\] 
and 
\eq{split_int} 
| D_{z}^{k} \mc{K} f (z)|^{p} \leq \left[ \int_{\U \sm\overline \Om} \del(\zeta)^{(1-p) \all} |A_k(z, \zeta)| |f(\zeta)|^{p} \, dV(\zeta)  \right] \left[ \int_{\U \sm\overline \Om} \del( \zeta)^{\all} |A_k (z, \zeta)| \, dV(\zeta)  \right]^{\frac{p}{p'}}.  
\eeq 
for some $\all$ to be chosen. 
Now $A_k$ is a sum of the form (see \eqref{Am})
\[
  A_k (z, \zeta) = \frac{P_{1} (z, \zeta)}{\Phi^{n-l+\mu_1} (z, \zeta)} \pa_{z}^{\mu_2} \left\{ \ti{N}_{1- 2l} (\zeta - z) \right\}, \quad 1 \leq l \leq n-1, \quad \mu_1 + \mu_2 \leq k. 
\]
By the same reasoning as before, it suffices to estimate the term for $l = n-1$, $\mu_1=k$, namely we have 
\[
|A_k(z, \zeta)| \lesssim \frac{1}{|\Phi^{k +1} (z, \zeta)| |\zeta -z|^{2n-3}}.  
\]
By a partition of unity in both $z$ and $\zeta$ space, we can assume that $\supp_z A_k(z, \zeta) \seq V \cap \Om $ and $\supp_\zeta f \seq V \sm\overline \Om$, where $V$ is some small neighborhood of a fixed point $\zeta_0 \in b \Om$. 

Now as $| \Phi (z, \zeta) | < C |z-\zeta|$, we can assume that $|\Phi (z, \zeta)| < 1$ for any $z, \zeta \in V$. Hence for $k \leq m$, 
\[
   |A_k (z, \zeta)| \leq \frac{C}{|\Phi^{m +1} (z, \zeta)| |\zeta -z|^{2n-3}}, \quad z, \zeta \in V.  
\] 
In view of \eqref{split_int}, the rest of proof is identical to that of previous theorem.
\end{proof}

Together with the Hardy-Littlewood lemmas from Section 4, we can prove the gain of $1/2$ derivative in Sobolev space for the operator $\Hc_q^1$. 
\begin{prop} \label{H1_prop}  
Let $\Om \seq \C^{n}$ be a bounded strictly pseudoconvex domain with $C^{2}$ boundary. For $q \geq 1$, let $\Hc_q \var$ be given by \eqref{hom_opt}, where the extension operator $\Ec$ is given by \eqref{E_def}. Suppose $\var \in H^{s,p}_{(0,q)}(\Om)$ with $s >\frac{1}{p}$, then $\Hc_q^1 \var \in H^{s+\yh,p}_{(0,q-1)}(\Om)$. 

\end{prop}\begin{proof}
We divide into cases. 

\noindent
\tit{Case 1:} $s = \frac{2k+1}{2}$, $k \in \Z^+$.

We have $s+ \yh \in \Z^{+} $ and also $s > \frac{1}{p}$. Take $m$ to be any integer greater than $s+\yh$. By Propositions \ref{loder_est} and \ref{H-L_int} we obtain 
\begin{align*} 
\| \Hc_q^1  \var   \|_{W^{s+\yh, p}(\Om)} \lesssim\sum_{|\gamma|\le k}\| \del^{m-s-\yh} D^\gamma\Hc_q^1  \var  \|_{L^p (\Om) }  
\lesssim \| \var \|_{H^{s, p} (\Om) }. 
\end{align*}
Since $\| \Hc_q^1  \var   \|_{W^{s+\yh, p}(\Om)} = \| \Hc_q^1  \var   \|_{H^{s+\yh, p}(\Om)}$ for $s+\yh$ a positive integer, the result follows.

\noindent
\tit{Case 2:} $s \in \left( \frac{1}{p}, \frac{3}{2} \right)$ and $s + \yh \in [0,1]$.

We apply \rp{H-L_lo} \ref{Item::H-L_lo::1} and \rp{loder_est} for $m=1$ to get
\[
\| \Hc_q^1  \var  \|_{H^{s+\yh, p} (\Om)} \leq \| \del^{1- s - \yh} \Hc_q^1  \var  \|_{L^p (\Om)} + \| \del^{1- s - \yh} D \Hc_q^1  \var  \|_{L^p (\Om)} 
\leq C \| \var \|_{H^{s,p}(\Om)}. 
\]
\tit{Case 3:}  $s \in \left( \frac{1}{p}, \frac{3}{2}  \right)$ and $s+\yh \in [1,2]$. 

We apply \rp{H-L_lo} \ref{Item::H-L_lo::2} and \rp{loder_est} for $m=2$ to get
\begin{align*} 
\| \Hc_q^1  \var  \|_{H^{s+\yh, p}(\Om)} 
&\leq C \| \Hc_q^1  \var  \|_{L^p (\Om) } + \| \del^{2-s-\yh} D \Hc_q^1  \var  \|_{L^p (\Om) }  + \| \del^{2-s-\yh} D^2 \Hc_q^1  \var  \|_{L^p (\Om) }
\\ &\leq C \| \var \|_{H^{s, p} (\Om) }. 
\end{align*}
Finally the remaining cases can be done by interpolation. 
\end{proof}

By combining \rp{H0_prop} and \rp{H1_prop} we obtain \rt{H_thm}.

\section{$\La^r$ estimate for $ r > 0$}  

Let $r>0$ and let $\Omega \seq \C^n$ be a bounded Lipschitz domain. We denote by $\Lambda^r(\Omega)$ the space of H\"older-Zygmund functions of order $r$ up to the boundary.

We first recall the interpolation result of  H\"older-Zygmund spaces.
\begin{prop}[Complex interpolation of $\Lambda^r$-spaces]
Let $r_0,r_1>0$ and let $\Omega\subset\C^n$ be a bounded Lipschitz domain. For $0<\theta< 1$, let $r_\theta=(1-\theta) r_0+\theta r_1$. Then $[\Lambda^{r_0}(\Omega),\Lambda^{r_1}(\Omega)]_\theta=\Lambda^{r_\theta}(\Omega)$.
\end{prop}
The proof of is the combination of \cite[Theorem 6.4.5(6)]{B-L76} and \cite[Theorems 1.110 and 1.122]{Tr06}.

\medskip
In \cite{Gong19}, Gong constructed a solution operator $\Sc_q \var$ to $\dbar u = \var$ which maps any $(0,q)$ form $\var \in \La^r(\Om)$ to  a $(0, q-1)$ form in $\La^{r+\yh} (\Om)$ for all $r > 1$. We now extend this result to all $r > 0$. 

First let us recall the classical Hardy-Littlewood lemma for H\"{o}lder continuous functions. 

\begin{lemma}\label{H-L_Holder} 
	Let $\Om$ be a bounded Lipschitz domain in $\R^N$ and let $\del (x)$ denote the distance function from $x$ to the boundary of $\Om$. If $u$ is a $C^1$ function in $\Om$ and there exists an $0 < \all < 1$ and $C>0$ such that
	\[
	|D u (x)| \leq C \del (x)^{-1 + \all} \quad \text{for every $x \in \Om$}, 
	\] 
	then $u \in \La^{\all}(\Om)$. 
\end{lemma}
The reader can refer to \cite[p.~345]{C-S01} for a proof. 

\begin{thm} \label{H_thm2}  
    Let $\Omega\seq\C^n$ be a bounded strictly pseudoconvex domain with $C^2$ boundary. Let $1\le q\le n$ and let $\Hc_q$ be given by formula \eqref{hom_opt}, where the extension operator $E$ is defined by formula \eqref{E_def}. Then for any $r > 0$, $\Hc_q$ is a bounded linear operator $\Hc_q: \La^{r}_{(0,q)} (\Om)\to \La^{r+ \yh}_{(0,q-1)}  (\Om)$.  
\end{thm}  
\begin{proof}

For $\var \in \La^r(\Omega)$ with $r > 1$, see \cite{Gong19}, where the homotopy operator is defined by formula \eqref{hom_opt} in the classical sense. We note that Gong used a different extension operator than ours, but in the case $r > 1$ the proofs work the same since the only property of the extension operator used in his proof is the fact $[D, \Ec]: \La^r(\Omega) \to \La^{r-1}(\Omega)$ for $r>1$, which obviously holds for our extension operator as well by Proposition \ref{Prop::ext_op} \ref{Item::ext_op::Hold}. 
	
By interpolation we only need to prove for $0< r< \yh$. 
In view of \rl{H-L_Holder}, it suffices to show that  
\[
  \sup_{z \in \Om} \del(z)^{1-(r+\yh)} \left| D \Hc^1_q \var (z) \right| \leq C \| \var \|_{\La^r (\Om) }. 
\]
We have 
\[
  \del(z)^{1-(r+\yh)}  | D \Hc^1_q  \var (z) | = \del(z)^{ \yh-r} \left| \int_{\U \sm\overline \Om} D_z K^{01}_{0,q} (z, \zeta) \we [\dbar, \Ec] \var (\zeta) \, dV(\zeta) \right|. 
\]
Write $A_1(z, \zeta) = D_z K^{01}_{0, q} (z, \zeta)$ and $f = [\dbar, \Ec] \var$. By \eqref{commest_Holder}, we get $\| \del^{1-r} f \|_{L^{\infty}(U \sm\overline \Om)} \leq \| f \|_{\La^r (\Om)} $, so 
\begin{align*} 
  \del(z)^{1-(r+\yh)}  | D \Hc^1_q  \var (z) | 
  &= \del(z) ^{\yh - r} \int_{\U \sm\overline \Om} | A_1 (z, \zeta)|  | f(\zeta)| \, dV(\zeta)  
\\ & = \del(z) ^{\yh - r} \int_{\U \sm\overline \Om} \del (\zeta)^{r-1} | A_1 (z, \zeta)|  \del (\zeta)^{1-r} |f(\zeta)| \, dV(\zeta)  
	\\& \leq \del(z)^{\yh - r}  \left(   \int_{\U \sm\overline \Om} \del (\zeta)^{r-1} | A_1 (z, \zeta)|\, dV(\zeta)  \right) \|f \|_{\La^r (\Om)}
\\ &\leq C \|f \|_{\La^r (\Om)}, 
\end{align*} 
where in the last inequality we apply \rl{intestle} for $\all = r-1$ and $\beta = 0$ (which is possible since $0 < r < \yh$ and $-1 < \all < \beta + \yh$) to get   
	\begin{align*}
	\int_{\U \sm\overline \Om} \del (\zeta)^{r-1} | A_1 (z, \zeta)| \, dV(\zeta)
	&\leq C  \int_{s_1=0}^{1} \int_{s_2=0}^{1} \int_{t =0}^{1} \frac{s_1^{r-1} t^{2n-3}\, ds_1 \, ds_2  \, dt}{(\del (z) + s_1 + s_2 + t^{2})^{2} (\del(z)+s_1+s_2+t)^{2n-3}} 
	\\
	&\leq C \del (z)^{r - \yh}.
	\end{align*}
	This completes the proof.
\end{proof}

\appendix
\section{An equivalent norm property}
\begin{prop}\label{Prop::EqvNormC2}
	Let $\Omega\seq\R^N$ be a bounded $C^2$-domain and let $1<p<\infty$.  
	\begin{enumerate}[(i)]
		\item\label{Item::EqvNormC2::02} For $0<s<2$, $H^{s,p}(\Omega)$ has equivalent norm
		\begin{equation*}
		\|f\|_{H^{s,p}(\Omega)}\approx\|f\|_{H^{s-1,p}(\Omega)}+\|Df\|_{H^{s-1,p}(\Omega)}.
		\end{equation*}
		\item\label{Item::EqvNormC2::12}  For $1<s<2$, $H^{s,p}(\Omega)$ has equivalent norm
		\begin{equation*}
		\|f\|_{H^{s,p}(\Omega)}\approx\|f\|_{H^{s-2,p}(\Omega)}+\|Df\|_{H^{s-2,p}(\Omega)}+\|D^2f\|_{H^{s-2,p}(\Omega)}.
		\end{equation*}
	\end{enumerate}
\end{prop}

Note that the above results are known for $C^\infty$-domain (see \cite[Theorem 3.3.5(ii)]{Tr83}). We shall prove it by reducing to the smooth case.

To prove Proposition \ref{Prop::EqvNormC2} we first need a lemma. 

\begin{lemma}\label{Lem::EqvNormC2Lem}
	Let $\Phi:\R^N\to\R^N$ be a $C^2$-diffeomorphism such that $D\Phi$ and $D\Phi^{-1}$ both have bounded $C^1$ norms. Then
	\begin{enumerate}[(i)]
	  \item\label{Item::EqvNormC2Lem::Mult} $\tilde f\mapsto\tilde fD\Phi$ defines a bounded linear map $H^{s,p}(\R^N)\to H^{s,p}(\R^N)$ for all $1<p<\infty$ and $-1\le s\le 1$.
		\item\label{Item::EqvNormC2Lem::Comp} $\tilde f\mapsto\tilde f\circ\Phi$ defines a bounded linear map $H^{s,p}(\R^N)\to H^{s,p}(\R^N)$ for all $1<p<\infty$ and $-1\le s\le 2$.
	\end{enumerate}

\end{lemma}

See also \cite[Theorem 4.3.2]{Tr92} for \ref{Item::EqvNormC2Lem::Comp}. Note that in the reference the result is only proved for the range $N(\frac1p-1)<s<2$, which is not enough for us if $\frac1p-1<-\frac1N$.
\begin{proof}
Clearly the map $[\tilde f\mapsto\tilde f D\Phi]:W^{1,p}(\R^N)\to W^{1,p}(\R^N)$ is bounded linear, and $[\tilde f\mapsto\tilde f D\Phi]:W^{1,p'}(\R^N)\to W^{1,p'}(\R^N)$ is also bounded linear.
	
Since the product operator is self-adjoint and $H^{-1,p}(\R^N)=W^{1,p'}(\R^N)'$(\rp{Prop::DualSpace}), the map $[\tilde f\mapsto\tilde f D\Phi]:H^{-1,p}(\R^N)\to H^{-1,p}(\R^N)$ is also bounded linear.
	
By using interpolation (\rp{cinterpol} and \rp{interpol_space}) , we get $[\tilde f\mapsto\tilde f D\Phi]:H^{s,p}(\R^N)\to H^{s,p}(\R^N)$ for all $-1\le s\le 1$. This finishes the proof of \ref{Item::EqvNormC2Lem::Mult}.
	
For \ref{Item::EqvNormC2Lem::Comp}, clearly $[\tilde f\mapsto\tilde f\circ\Phi]:W^{k,p}(\R^N)\to W^{k,p}(\R^N)$ are bounded linear for $k=0,1,2$ since $D\Phi,D^2\Phi$ are bounded. Since $W^{k,p}=H^{k,p}$ we have the boundedness for $s=0,1,2$.
	
Using change of variables, the adjoint map of $\tilde f\mapsto\tilde f\circ\Phi$ is $\tilde g\mapsto|\det D\Phi^{-1}|\cdot(\tilde g\circ\Phi^{-1})$. 
Clearly $[\tilde g\mapsto \tilde g\circ\Phi^{-1}]:W^{1,p'}(\R^N)\to W^{1,p'}(\R^N)$ is bounded linear. Since $\det D\Phi^{-1}$ is bounded $C^1$ and non-vanishing everywhere, we have $|\det D\Phi^{-1}|\in C^1(\R^N)$. Therefore the map $[\tilde g\mapsto|\det D\Phi^{-1}|\cdot(\tilde g\circ\Phi^{-1})]:W^{1,p'}(\R^N)\to W^{1,p'}(\R^N)$ is bounded linear.
	
Since $H^{-1,p}(\R^N)=W^{1,p'}(\R^N)'$, taking the adjoint back we get the boundedness $[\tilde f\mapsto\tilde f\circ\Phi]:H^{-1,p}(\R^N)\to H^{-1,p}(\R^N)$, which proves the case $s=-1$.
Finally by interpolation, we get $[\tilde f\mapsto\tilde f\circ\Phi]:H^{s,p}(\R^N)\to H^{s,p}(\R^N)$ for all $-1\le s\le 2$.
\end{proof}

\begin{proof}[Proof of Proposition \ref{Prop::EqvNormC2}]
Recall the definition $\| f  \|_{H^{s,p} ( \Om)} = \min_{\wti{f}|_{\Om} = f} \|  \wti{f} \|_{H^{s,p} (\R^N)}$. The ``$\gtrsim$''-part follows from the equivalent norm in $\R^N$ (see \cite[Theorem 2.3.8(ii)]{Tr83}) and the fact that if $\tilde f\in H^{s,p}(\R^N)$ extends $f\in H^{s,p}(\Omega)$ then $D^\alpha\tilde f\in H^{s-1,p}(\R^N)$ extends $D^\alpha f$. Also see the proof of \cite[Theorem 3.3.5(ii)]{Tr83}.
	
We now prove the ``$\lesssim$''-part. It suffices to prove \ref{Item::EqvNormC2::02}, since \ref{Item::EqvNormC2::12} follows by applying \ref{Item::EqvNormC2::02} twice.
	
By \cite[Theorem 3.3.5(ii)]{Tr83}, the result holds on the half space $\R^N_+=\{x_N>0\}$, namely, for every  $r\in\R$ and $1<p<\infty$, the relation
	\begin{align}\label{Eqn::EqvNormC2::Tmp}
	\|g\|_{H^{r,p}(\R^N_+)}&\approx_{r,p}\|g\|_{H^{r-1,p}(\R^N_+)}+\|Dg\|_{H^{r-1,p}(\R^N_+)}, 
	\end{align}
	holds for all $g\in H^{r,p}(\R^N_+)$ supported in $\B^N\cap\overline{\R^N_+}$. Here $\B^N$ is the unit ball in $\R^N$. 
	
By partition of unity, we can find the following: 
	\begin{itemize}
		\item Open sets $(U_\nu)_{\nu=1}^M$ such that $b\Omega\seq\bigcup_{\nu=1}^MU_\nu$.
		\item Functions $\chi_0\in C_c^\infty(\Omega)$, $\chi_\nu\in C_c^\infty(U_\nu)$ for $1\le \nu\le M$ such that $\sum_{\nu=0}^M\chi_\nu|_\Omega\equiv1$.
		\item $C^2$-maps $\Phi_\nu:\R^N\to\R^N$ for $1\le \nu\le M$ such that $\Phi_\nu(\B^N)=U_\nu$, $\Phi_\nu(\B^N\cap\R^N_+)=U_\nu\cap\Omega$ and $D\Phi_\nu,D\Phi_\nu^{-1}$ have bounded $C^1$ norm.
	\end{itemize} 
	Therefore by Lemma \ref{Lem::EqvNormC2Lem} \ref{Item::EqvNormC2Lem::Comp},
	\begin{align*}
	\|f\|_{H^{s,p}(\Omega)}\le\sum_{\nu=0}^M\|\chi_\nu f\|_{H^{s,p}(\Omega)}\lesssim\|\chi_0 f\|_{H^{s,p}(\R^N)}+\sum_{\nu=1}^M\|(\chi_\nu f)\circ\Phi_\nu\|_{H^{s,p}(\R^N_+)}.
	\end{align*}
	By  \cite[Theorem 2.3.8(ii)]{Tr83} we have
	\begin{align*}
	\|\chi_0 f\|_{H^{s,p}(\R^N)}
	&\approx\|\chi_0 f \|_{H^{s-1,p}(\R^N)}+\|D (\chi_0 f)\|_{H^{s-1,p}(\R^N)}
	\\
	& \leq \|\chi_0 f \|_{H^{s-1,p}(\R^N)}+\|f D \chi_0 \|_{H^{s-1,p}(\R^N)}+\| \chi_0 D f\|_{H^{s-1,p}(\R^N)}
	\\
	&\lesssim \|f\|_{H^{s-1,p}(\Omega)}+\| D f\|_{H^{s-1,p}(\Om)}.
	\end{align*}
	
For $1\le \nu\le M$, we apply Lemma \ref{Lem::EqvNormC2Lem} with $-1 \leq s- 1 \leq 1$, 
\begin{align*}
\|(\chi_\nu f)\circ\Phi_\nu\|_{H^{s,p}(\R^N_+)}
  & \lesssim\|(\chi_\nu f)\circ\Phi_\nu\|_{H^{s-1,p}(\R^N_+)}+\|D((\chi_\nu f)\circ\Phi_\nu)\|_{H^{s-1,p}(\R^N_+)}
    \\
&\lesssim \|(\chi_\nu f)\circ\Phi_\nu\|_{H^{s-1,p}(\B^N\cap\R^N_+)}+\|((D(\chi_\nu f))\circ\Phi_\nu) \cdot D\Phi_\nu\|_{H^{s-1,p}(\R^N_+)}
	\\
&\lesssim \|(\chi_\nu f)\circ\Phi_\nu\|_{H^{s-1,p}(\B^N\cap\R^N_+)}+\|(D(\chi_\nu f)\circ\Phi_\nu\|_{H^{s-1,p}(\R^N_+)}
	\\
&\lesssim \|\chi_\nu f\|_{H^{s-1,p}(\Omega)}+\|D(\chi_\nu f)\|_{H^{s-1,p}(\Omega)} 
\lesssim\| f\|_{H^{s-1,p}(\Omega)} +\|Df\|_{H^{s-1,p}(\Omega)}.
\end{align*}
	
By taking sum over $0\le \nu\le M$ we complete the proof.
\end{proof}

\bibliographystyle{amsalpha}
\bibliography{master}

\end{document}